\documentclass[11pt]{article}
\usepackage{graphicx}
\usepackage{amssymb}
\usepackage{color}

\renewcommand{\theequation}{\arabic{equation}}
\newcommand{\D}{\displaystyle}

\newcommand{\Bx}{\mathbf{x}}
\newcommand{\Ox}{\Omega_\Bx}
\newcommand{\Dx}{\Delta_{\Bx}}

\newcommand{\p}{\partial}
\newcommand{\al}{\alpha}

\newcommand{\myref}[1]{(\ref{#1})}

\newcommand{\mathd}{\mathrm{d}}

\newtheorem{theorem}{\textbf{Theorem}}[section]
\newtheorem{lemma}{\textbf{Lemma}}[section]
\newtheorem{remark}{\textbf{Remark}}[section]
\newtheorem{proposition}{\textbf{Proposition}}[section]

\newtheorem{definition}{\textbf{Definition}}[section]

\newenvironment{proof}{\noindent\textbf{Proof\ }}{\hspace*{\fill}$\Box$\medskip}
\begin{document}
\title{On Singularity Formation of a 3D Model for Incompressible Navier-Stokes Equations}
\author{Thomas Y. Hou\thanks{Applied and Comput. Math, Caltech,
Pasadena, CA 91125. {\it Email: hou@acm.caltech.edu.}} \and
Zuoqiang Shi\thanks{Applied and Comput. Math, Caltech,
Pasadena, CA 91125. {\it Email: shi@acm.caltech.edu.}} \and
Shu Wang\thanks{College of Applied Sciences, Beijing University of
Technology, Beijing 100124, China.
{\it Email: wangshu@bjut.edu.cn}}
}
\date{\today}
\maketitle

\begin{abstract}

We investigate the singularity formation of a 3D model that
was recently proposed by Hou and Lei in \cite{HouLei09a} for
axisymmetric 3D incompressible Navier-Stokes equations with swirl.
The main difference between the 3D model of Hou and Lei and the
reformulated 3D Navier-Stokes equations is that the convection
term is neglected in the 3D model. This model shares many
properties of the 3D incompressible Navier-Stokes equations.
One of the main results of this paper is that we prove rigorously
the finite time singularity formation of the 3D inviscid model
for a class of initial boundary value problems with smooth initial
data of finite energy. We also prove the global regularity
of the 3D inviscid model for a class of small smooth initial data.

\end{abstract}


\section{Introduction}

The question of whether a solution of the 3D incompressible
Navier-Stokes equations can develop a finite time singularity from
smooth initial data with finite energy is one of the most outstanding
mathematical open problems \cite{Fefferman,MB02,Temam01}.
Most regularity analysis for the 3D Navier-Stokes equations relies
on energy estimates. Due to the incompressibility condition, the
convection term does not contribute to the energy norm of the
velocity field or any $L^p$ ($1 < p \leq \infty$) norm of the
vorticity field. As a result, the main effort has been to use the
diffusion term to control the nonlinear vortex stretching term
without making use of the convection term explicitly.

In a recent paper by Hou and Lei \cite{HouLei09a}, the authors
investigated the effect of convection by constructing
a new 3D model for axisymmetric 3D incompressible Navier-Stokes
equations with swirl. Specifically, their 3D model is given below:
\begin{eqnarray}
\label{model-u1}
&&\partial_tu_1  =
  \nu\big(\partial_r^2 + \frac{3}{r}\partial_r
 + \partial_z^2\big)u_1 + 2\partial_z\psi_1u_1,\\
\label{model-w1}
&&\partial_t\omega_1 = \nu\big(\partial_r^2 + \frac{3}{r}\partial_r
  + \partial_z^2\big)\omega_1 + \partial_z\big((u_1)^2\big),\\
\label{model-psi1}
&&- \big(\partial_r^2 + \frac{3}{r}\partial_r
  + \partial_z^2\big)\psi_1 = \omega_1.
\end{eqnarray}
Note that \myref{model-u1}-\myref{model-psi1} is already a closed system.
The only difference between this 3D model and the reformulated
Navier-Stokes equations is that the convection term is neglected
in the model. If one adds the convection term back to the left hand side
of \myref{model-u1} and \myref{model-w1}, one would recover the full
Navier-Stokes equations. This model preserves almost all the
properties of the full 3D Navier-Stokes equations, including the energy
identity for smooth solutions of the 3D model and the divergence free
property of the reconstructed 3D velocity field given by $u^\theta = ru_1$,
$ u^r = - \partial_z(r\psi_1)$,
$u^z = \frac{1}{r}\partial_r(r^2\psi_1)$.
Moreover, they proved the corresponding
non-blowup criterion of Beale-Kato-Majda type \cite{BKM84}
as well as a non-blowup criterion of Prodi-Serrin type
\cite{Prodi,Serrin} for the model. In a subsequent paper,
they proved a new partial regularity result for the model
\cite{HouLei09b} which is an analogue of the
Caffarelli-Kohn-Nirenberg theory \cite{CKN82} for the full
Navier-Stokes equations.

Despite the striking similarity at the theoretical level between
the 3D model and the Navier-Stokes equations, the former seems
to have a very different behavior from the full Navier-Stokes equations.
In \cite{HouLei09a}, the authors presented numerical evidence
which supports that the 3D model may develop a potential finite time
singularity. They further studied the mechanism that leads to
these singular events in the 3D model. On the other hand,
the Navier-Stokes equations with the same initial data seems
to have a completely different behavior.

One of the main results of this paper is that we prove rigorously
the finite time singularity formation of this 3D model for
a class of initial boundary value problems with smooth
initial data of finite energy. In our analysis, we focus on
the inviscid version of the 3D model and consider the
initial boundary value problem of the generalized 3D model
which has the following form \cite{HouLei09a} (we drop the
subscript 1 and substitute \myref{model-psi1} into \myref{model-w1}):
\begin{eqnarray}
\label{model-u-mixed}
u_t &=&2u\psi_z,\\
\label{model-psi-mixed}
-\Delta \psi_t&=&\left(u^2\right)_z,
\end{eqnarray}
where $\Delta$ is a $n$-dimensional Laplace operator
with $(\Bx,z) \equiv (x_1, x_2, ...,x_{n-1},z)$. Our results
in this paper apply to any dimension greater than or equal
to two ($n \geq 2$). To simplify
our presentation, we only present our analysis for $n=3$.
We consider the generalized 3D model in both a bounded
domain and in a semi-infinite domain with a mixed Dirichlet
Robin boundary condition.
The main result of this paper is the following theorem.
\begin{theorem}
\label{theorem bounded mix-1}
Let $\Ox = (0,a)\times (0,a)$, $\Omega = \Ox \times (0,b)$ and
$\Gamma =\{ (\Bx,z) \; | \; \Bx \in \Ox, \; z=0\}$.
Assume that the initial condition $u_0 > 0$ for $(\Bx,z) \in \Omega$,
$u_0|_{\p \Omega}=0$, $u_0\in H^2(\Omega)$,
$\psi_{0}\in H^3(\Omega)$ and satisfies \myref{psi-bc}.
Moreover, we assume that
$\psi$ satisfies the following mixed
Dirichlet Robin boundary condition:
\begin{equation}
\label{psi-bc}
\psi|_{\partial \Omega \backslash \Gamma} = 0, \quad (\psi_z + \beta \psi )|_\Gamma = 0,
\end{equation}
with
$\beta>\frac{\sqrt{2}\pi}{a}
\left(\frac{1+e^{-2\pi b/a}}{1-e^{-2\pi b/a}}\right)$.
Define $\phi(x_1,x_2,z)=\left (\frac{e^{-\alpha (z-b)}+e^{\alpha (z-b)}}{2}\right )
\sin \left (\frac{\pi x_1}{a}\right ) \sin \left (\frac{\pi x_2}{a}\right )$
where $\al$ satisfies $0<\al<\sqrt{2}\pi/a$ and
$ 2\left(\frac{\pi}{a}\right)^2
\frac{e^{\al b}-e^{-\al b}}{\al(e^{\al b}+e^{-\al b})} = \beta$.
If $u_0$ and $\psi_0$ satisfy the following condition:
\begin{eqnarray}
\label{eqn-IC}
\int_\Omega (\log u_0) \phi d \Bx dz  > 0,\quad
\int_\Omega \psi_{0z}\phi d\Bx dz   > 0,
\end{eqnarray}
then the solution of the 3D inviscid model
\myref{model-u-mixed}-\myref{model-psi-mixed} will develop a
finite time singularity in the $H^2$ norm.
\end{theorem}

The analysis of the finite time singularity for the 3D model is
rather subtle. The main technical difficulty is that this is a
multi-dimensional nonlinear nonlocal system. Currently, there is
no systematic method of analysis to study singularity formation
of a nonlinear nonlocal system. The key issue is
under what condition the solution $u$ has a strong alignment
with the solution $\psi_z$ dynamically. If $u$ and $\psi_z$ have a strong
alignment for long enough time, then the right hand side of
the $u$-equation would develop a quadratic nonlinearity dynamically,
which would lead to a finite time blowup. Note that $\psi$ is coupled
to $u$ in a nonlinear and nonlocal fashion. It is not clear whether
$u$ and $\psi_z$ will develop such a nonlinear alignment dynamically.
As a matter of fact, not all initial boundary conditions of
the 3D model would lead to finite time blowup. One of the interesting
results we obtain in this paper is that we prove the global
regularity of the 3D inviscid model for a class of small
initial data with an appropriate boundary condition. We would
like to point out that since there is no viscosity in the 3D
inviscid model, such global regularity result is still interesting
even though some smallness condition is imposed on the initial data.
We note that there is currently no corresponding global regularity result
for the incompressible 3D Euler equation even with small initial data.

One of the main contributions of this paper is that we introduce
an effective method of analysis to study singularity formation
of this nonlinear nonlocal multi-dimensional system. There are
several important steps in our analysis. The first one is that
we reformulate the $u$-equation so that the right hand side of
the reformulated $u$-equation becomes linear. This is accomplished
by dividing both sides of \myref{model-u-mixed} by $u$ and
introducing $\log(u)$ as a new variable. This is possible
since $u_0 > 0$ in $\Omega$ implies that $u > 0$ in $\Omega$
as long as the solution remains smooth.
The reformulated system now has the form:
\begin{eqnarray}
\label{model-logu}
\left (\log(u) \right )_t &=&2\psi_z, \quad (\Bx,z) \in \Omega ,\\
\label{model-Dpsi}
-\Delta \psi_t&=&\left(u^2\right)_z.
\end{eqnarray}
This idea is similar in spirit to the renormalized Boltzmann
equation introduced by DiPerna and Lions in their study of
the global renormalized weak solution of the Boltzmann equations
\cite{DL89}. The second step
is to work with the weak formulation of the reformulated model
\myref{model-logu}-\myref{model-Dpsi} by introducing an
appropriately chosen weight function $\phi$ as our test function.
How to choose this weight function $\phi$ is crucial in obtaining
the nonlinear estimate that is required to prove finite
time blowup of the nonlocal system. Guided by our analysis,
we look for a smooth and positive eigen-function in $\Omega$
that satisfies the following two conditions simultaneously:
\begin{eqnarray}
\label{eqn-eigenfunction}
-\Delta \phi = \lambda_1 \phi , \quad
\partial_z^2 \phi = \lambda_2 \phi ,
\quad \mbox{for} \; \mbox{some} \;
\lambda_1, \lambda_2 > 0, \quad (\Bx,z) \in \Omega .
\end{eqnarray}
The function $\phi$ defined in Theorem 1.1 satisfies
both of these conditions. We remark that such eigen-function
exists only for space dimension greater than or equal to two.
In the third step, we
multiply $\phi$ to \myref{model-logu} and $\phi_z$
to \myref{model-Dpsi}, integrate over $\Omega$,
and perform integration by parts. We obtain by
using \myref{eqn-eigenfunction} that
\begin{eqnarray}
\label{model-logu-int}
\frac{d}{dt} \int_\Omega (\log u) \phi \mathd {\Bx}\mathd z  &=&
2 \int_\Omega \psi_z \phi \mathd {\Bx}\mathd z  ,\\
\label{model-Dpsi-int}
\lambda_1 \frac{d}{dt} \int_\Omega \psi_z \phi
\mathd {\Bx}\mathd z &=&
\lambda_2 \int_\Omega u^2 \phi \mathd {\Bx}\mathd z .
\end{eqnarray}
All the boundary terms resulting from integration
by parts vanish by using the boundary condition
of $\psi$, the fact that $u|_{z=0} = u|_{z=b}=0$,
 the property of our eigen-function $\phi$, and
the specific choice of $\alpha$ defined in Theorem 1.1.
Substituting \myref{model-Dpsi-int} into \myref{model-logu-int}
gives the crucial estimate for our blowup analysis:
\begin{eqnarray}
\label{model-logu-int-tt}
\frac{d^2}{dt^2} \int_\Omega (\log u) \phi \mathd {\Bx}\mathd z =
\frac{2 \lambda_2}{\lambda_1} \int_\Omega u^2 \phi
\mathd {\Bx}\mathd z .
\end{eqnarray}
Further, we note that
\begin{eqnarray}
\label{eqn-logu1}
&&\int_\Omega \log(u) \phi \mathd {\Bx}\mathd z \leq
\int_\Omega (\log(u))^+ \phi \mathd {\Bx}\mathd z
\leq \int_\Omega u \phi \mathd {\Bx}\mathd z \nonumber \\
\label{eqn-logu2}
&& \leq \left (\int_\Omega \phi \mathd {\Bx}\mathd z \right )^{1/2}
\left (\int_\Omega \phi u^2 \mathd {\Bx}\mathd z \right )^{1/2} \equiv
\frac{2a}{\pi\sqrt{\alpha }}
\left (\int_\Omega \phi u^2 \mathd {\Bx}\mathd z \right )^{1/2}.
\end{eqnarray}
Integrating \myref{model-logu-int-tt} twice in time
and using \myref{eqn-logu1}-\myref{eqn-logu2},
we establish a sharp nonlinear
dynamic estimate for $(\int_\Omega \phi u^2 \mathd {\Bx}\mathd z )^{1/2}$,
which enables us to prove finite time blowup of the 3D model.

%This method of analysis is quite robust and captures very well
%the nonlinear interaction of the multi-dimensional nonlocal
%system.  As a result, it provides
%a very effective method to analyze the finite time blowup
%of the 3D model and gives a relatively sharp blowup condition
%on the initial and boundary values for the 3D model.
Another interesting result is that we prove the finite
time blowup of the 3D model with partial viscosity. Under similar
assumptions on $u_0$, $\psi_0$ and $\omega_0$ as in the
inviscid case and
by assuming that $\omega$ satisfies a boundary condition
similar to $\psi$, we can prove that the 3D model with
partial viscosity
\begin{eqnarray}
\label{model-u-vis}
u_t &=&2u\psi_z,\\
\label{model-w-vis}
\omega_t&=&\left(u^2\right)_z + \nu \Delta \omega , \\
\label{model-psi-vis}
-\Delta \psi&=& \omega ,
\end{eqnarray}
develops a finite time singularity.

We also study singularity formation of the 3D model with
$\beta = 0$ in \myref{psi-bc}. This case is interesting
because the smooth
solution of the corresponding 3D model satisfies an energy
identity. In this case, we can establish a finite time
blowup under an additional condition :
\[
\int_0^a \int_0^a (\psi - \psi_0)|_\Gamma
\sin \left (\frac{\pi x_1}{a}\right )
\sin \left (\frac{\pi x_2}{a}\right )
d \Bx < c_0 \int_\Omega \psi_{0z} \phi \mathd {\Bx}\mathd z ,
\]
as long as the solution remains regular, where
$c_0>0$ depends only on the size of the domain.

We remark that although the 3D model using the mixed
Dirichlet Robin boundary condition with $\beta \neq 0$
does not conserve energy exactly, we prove that the energy
remains bounded as long as the solution is smooth and
$\beta < c_0$ for some $c_0 >0$. We also establish the
local well-posedness of the initial boundary problem
with the mixed Dirichlet Robin boundary condition.
Our numerical study shows that the energy is still bounded
up to the blowup time even if $\beta > c_0$. Our
numerical study also suggests that the nature of the
singularity in the case of $ \beta > c_0$ is
qualitatively similar to that in the case of $\beta < c_0$.

Study of singularity formation for various model equations for
the 3D Euler/Navier-Stokes equations or the surface quasi-geostrophic
equation has been investigated by a number of people, including
Constantin-Lax-Majda \cite{CLM85},
Constantin \cite{Constantin86}, DeGregorio
\cite{DeGregorio90,DeGregorio96}, Kerr \cite{Kerr93}, 
Caflisch-Siegel \cite{CS04}, Cordoba-Cordoba-Fontelos
\cite{CCF05}, Chae-Cordoba-Cordoba-Fontelos \cite{CCCF05},
Matsumotoa-Becb-Frisch \cite{MBF07}, Hou-Li \cite{HouLi06},
Li-Sinai \cite{LS08},
Li-Rodrigo \cite{LR09}, and
Hou-Li-Shi-Wang-Yu \cite{HLSWY09}. The
effect of convection has also been studied by Hou and Li in a
recent paper \cite{HouLi08} via a new 1D model.
They proved dynamic stability of this 1D model by exploiting
the nonlinear cancellation between the convection and the vortex
stretching term, and constructing
a Lyapunov function which gives rise to a global pointwise
estimate for the derivatives of the vorticity in their model.

We would like to point out that the study of \cite{HouLi08,HouLei09a}
is based on a reduced model for some special flow geometry.
One should not conclude that convection term could lead to
depletion of singularity of the Navier-Stokes equations in general.
It is possible that convection term may act as a destabilizing term
for a different flow geometry. One of the main findings of
\cite{HouLi08,HouLei09a} and the present paper is that convection
term carries important physical information that should not be
neglected in our analysis of the Navier-Stokes equations.
Since the behavior of the 3D model is very different from that
of the Navier-Stokes equations, it is important to develop
a method of analysis that could take into account the physical
significance of convection term in an essential way.

The rest of the paper is organized as follows. In Section 2, we
study the local well-posedness of the 3D inviscid model and some
properties of the model. In section 3, we prove the finite time
blowup of the 3D inviscid model with mixed Dirichlet and Robin
boundary conditions. In Section 4, we prove finite time blowup of
the 3D model with partial viscosity. Section 5 is devoted to
analyzing the finite time blowup of the 3D inviscid  model with
some conservative boundary conditions. In Section 6, we prove the
global regularity of the 3D inviscid model for a class of small
initial initial data with some appropriare boundary condition. A
technical lemma is proved in Appendix A.

\section{Properties of the 3D model}

\subsection{Local well-posedness in $H^s$}

In this section, we will establish the local well-posedness of the
initial boundary problem of the 3D model with the mixed
Dirichlet Robin boundary condition. We will present our
analysis for the semi-infinite domain using the Sobolev space
$H^s$. The same result is also true in a bounded domain.

Consider the 3D model with the following mixed initial
boundary condition:
\begin{eqnarray}
\label{model psi mixed}
&&\left\{\begin{array}{rcl}
u_t &=&2u\psi_z\\
-\Delta \psi_t&=&\left(u^2\right)_z
\end{array}\right., \quad (\Bx,z)\in \Omega=\Omega_\Bx\times(0,\infty)\\
\label{model psi mixed bc}
&&\psi|_{\p\Omega\backslash \Gamma}=0,
\quad \left(\psi_z+\beta \psi\right)|_{\Gamma}=0, \\
&& \psi|_{t=0}=\psi_0(\Bx,z),\quad u|_{t=0}=u_0(\Bx,z)\ge0
,\label{model psi initial}
\end{eqnarray}
where $\Bx=(x_1,x_2),\;\Ox=(0,a)\times(0,a),\;
\Gamma=\left\{(\Bx,z) \;|\; \Bx \in \Ox,\; z=0\right\}$
and $\Delta=\Delta_\Bx+\frac{\p^2}{\p z^2}=\frac{\p^2}{\p x_1^2}+\frac{\p^2}{\p x_2^2}+\frac{\p^2}{\p z^2}$.

The local well-posedness analysis depends on an important
property of the elliptic operator with the mixed
Dirichlet Robin boundary condition.

\begin{lemma}
\label{reisz}
There exists a unique solution $v\in H^{s}(\Omega)$
to the boundary value problem:
\begin{eqnarray}
\label{eqn-Laplace}
&&-\Delta v = f,\quad (\Bx,z)\in \Omega, \\
\label{mix bc}
&&v|_{\partial \Omega\backslash \Gamma} =0,
\quad (v_z+\beta v)|_\Gamma =0,
\end{eqnarray}
if $\beta \in S_\infty \equiv \{\beta \;|\;
\beta\ne \frac{\pi|k|}{a} \;\;\mbox{for all} \; k\in \mathbb{Z}^2\}$,
$f\in H^{s-2}(\Omega)$ with $s \ge 2$ and
$f|_{\p\Omega\backslash\Gamma}=0$. Moreover we have
\begin{eqnarray}
\|v\|_{H^s(\Omega)}\le C_s \|f\|_{H^{s-2}(\Omega)},
\end{eqnarray}
where $C_s$ is a constant depending on $s$, $|k|=\sqrt{k_1^2+k_2^2}$.
\end{lemma}
We defer the proof of Lemma \ref{reisz} to Appendix A.

\begin{remark}
We remark that we can prove the same result as in Lemma \ref{reisz}
for a bounded domain $\Omega = \Ox\times (0,b)$ with the same
boundary condition by assuming
that $\beta\in S_b$ where
\begin{equation}
\label{set-S}
S_b =\{\beta \;| \; \beta \ne \frac{\pi|k|}{a}, \;
\mbox{and}\; \beta \ne \frac{\pi|k|}{a}\left ( \frac{ 1+e^{-2|k|\pi b/a}}
{ 1-e^{-2|k|\pi b/a}}\right ) \mbox{for all} \; k\in \mathbb{Z}^2\}.
\end{equation}
\end{remark}

\begin{remark}
We would like to point out that regularity estimates for the
second order elliptic problem with mixed Dirichlet and
Robin boundary conditions
have been studied by Temam and Ziane \cite{RZ04} in
the context of geophysical flows. However,
there is an important difference between the case investigated
by Temam and Ziane and the case considered by us here. Although
the problem is formulated slightly differently, the case
considered by Temam and Ziane corresponds to the case of
$\beta < 0$ on $\Gamma$, which gives rise to a dissipative boundary
condition. The case of $\beta >0$ is the main focus of
our present study. This case is more difficult because
the boundary contribution from the Robin boundary
condition produces the wrong sign when we perform energy
estimates. In our analysis, we need to study the spectral
proprety of the differential operator and exclude an infinite
number of discrete eigenvalues from $\beta$ in order to
obtain well-posedness of the elliptic problem with this
mixed Dirichlet Robin boundary condition.
\end{remark}

\begin{definition}
\label{def-k}
Let $\mathcal{K}: H^{s-2}(\Omega)\rightarrow H^{s}(\Omega)$ be a linear operator defined as following:
\begin{eqnarray*}
  \mbox{for all}\quad f\in H^{s-2}(\Omega),\quad \mathcal{K}(f) \;\mbox{is the solution of the boundary value problem
\myref{eqn-Laplace}-\myref{mix bc}}.
\end{eqnarray*}
\end{definition}
It follows from Lemma \ref{reisz} that for any $f\in H^{s-2}(\Omega)$,
we have
\begin{eqnarray}
\label{bound k}
  \|\mathcal{K}(f)\|_{H^s(\Omega)}\le  C_s\|f\|_{H^{s-2}(\Omega)} .
\end{eqnarray}

We also need the following well-known Sobolev inequality \cite{Foland95}.
\begin{lemma}
\label{lemma hs}
Let $u,v \in H^s(\Omega)$ with
$s>3/2$. We have
\begin{eqnarray}
\|uv\|_{H^s(\Omega)}\le c \|u\|_{H^s(\Omega)}\|v\|_{H^s(\Omega)}.
\end{eqnarray}
\end{lemma}

Now we can state the local well-posedness result for the 3D model
with the mixed Dirichlet Robin boundary condition.
\begin{theorem}
\label{local wellposedness} Assume that $u_0 \in H^s(\Omega)$,
$\psi_0 \in H^{s+1}(\Omega)$ for some $s>3/2$, $u_0|_{\partial
\Omega}= 0 $ and $\psi_0$ satisfies \myref{model psi mixed bc}.
Moreover, we assume that $\beta \in S_\infty$ (or $S_b$) as
defined in Lemma \ref{reisz}.%, and $\psi$ satisfies the following
%mixed Dirichlet Robin boundary condition:
%\begin{equation}
%\label{BC-psi}
%\psi|_{\partial \Omega\backslash\Gamma} =0,\quad (\psi_z+\beta \psi)|_\Gamma=0.
%\end{equation}
Then there exists a finite time $T=T\left(\|u_0\|_{H^s(\Omega)},
\|\psi_{0}\|_{H^{s+1}(\Omega)}\right)>0$ such that the system
\myref{model psi mixed}-\myref{model psi initial}
%\myref{model-u-mixed}-\myref{model-psi-mixed} with boundary
%condition \myref{BC-psi}
has a unique solution, $u\in
C^1([0,T),H^s(\Omega))$ and $\psi\in C^1([0,T),H^{s+1}(\Omega))$.
\end{theorem}
\begin{proof}
Let $v=u^2$, then we obtain an equivalent system for
$v$ and $\psi$ as follows:
\begin{eqnarray}
\label{system-reform1}
v_t &=&4v\psi_z,\\
\label{system-reform2}
\psi_t&=&\mathcal{K}(v_z),
\end{eqnarray}
where $\mathcal{K}$ is defined in Definition \ref{def-k}.
To prove the local well-posedness of system
(\ref{system-reform1})-(\ref{system-reform2}),
we introduce the space
\[
V^{s+1}=\{\psi\in H^{s+1}(\Omega):\psi|_{\partial \Omega\backslash\Gamma} =0,
(\psi_z+\beta \psi)|_\Gamma=0\}.
\]
By the Trace Theorem \cite{Evans98},
the trace of $\psi$ and $\psi_z$ on $\partial \Omega$
is well defined since we assume that
$\psi\in H^{s+1}(\Omega)$ with $s > 3/2$.
Then we can write the system (\ref{system-reform1})-(\ref{system-reform2})
as an ODE in the Banach space $X:=H^s(\Omega) \times V^{s+1}(\Omega)$:
\begin{equation}
\label{new-system}
U_t = F(U),
\end{equation}
where $U = (U_1, U_2)= (v,\psi)$, $F(U) = (F_1(U),
F_2(U))=\left(4v\psi_z, \mathcal{K}(v_z)\right)$ and the norm
$\|\cdot\|_X$ of the space $X$ is defined as follows:
\begin{eqnarray*}
  \|U\|_X=\|U_1\|_{H^s(\Omega)}+\|U_2\|_{H^{s+1}(\Omega)}.
\end{eqnarray*}

We will use the well-known Picard theorem on a Banach space
(see e.g. Theorem 3.1 in \cite{MB02}) to prove the local well-posedness
of system (\ref{new-system}). In order to apply the Pichard theorem
on a Banach space, we need to check the following two conditions:
\begin{itemize}
\item[1.] $F$ maps $O\subset X$ to $X$, where $O$ is an open subset of $X$.
\item[2.] $F$ is locally Lipschitz continuous, i.e. for any $U\in O$, there exists $L>0$ and an open
neighborhood of $U$, $B_U\subset O$, such that
\begin{eqnarray*}
  \|F(\bar{U})-F(\tilde{U})\|_X\le L  \|\bar{U}-\tilde{U}\|_X,\quad \mbox{for all}\quad \bar{U}, \tilde{U} \in B_U.
\end{eqnarray*}
\end{itemize}
First, we choose the open set $O$ to be a bounded set defined as following:
\begin{eqnarray}
  \label{open-set}
  O=\left\{U\in X: \|U\|_X< M\right\},
\end{eqnarray}
where $M>0$ is a constant.

To verify the first condition, we obtain by
using estimate \myref{bound k} and Lemma \ref{lemma hs} that
\begin{eqnarray}
  \label{bounded-picard}
  \|F(U)\|_X&=&\|F_1(U)\|_{H^s}+\|F_2(U)\|_{H^{s+1}}\nonumber\\
&=& \|4U_1U_{2z}\|_{H^s}+\| \mathcal{K}(U_{1z})\|_{H^{s+1}}\nonumber\\
&\le& 4C_s \|U_1\|_{H^s} \|U_{2z}\|_{H^s}+ C_s\|U_{1z}\|_{H^{s-1}}\nonumber\\
&\le& 4C_s \|U_1\|_{H^s} \|U_2\|_{H^{s+1}}+ C_s\|U_{1}\|_{H^{s}}\nonumber\\
&\le& 4C_s\|U\|_X\left(1+\|U\|_X\right) < 4C_s M(1+M),
\end{eqnarray}
where $U_{iz} \equiv \left (U_i \right)_z $ ($i=1,2$).

Next, we show that $F$ is locally Lipschitz continuous.
For any $\bar{U}, \tilde{U}\in O$, we have by using
\myref{bound k} and Lemma \ref{lemma hs} that
\begin{eqnarray}
  \label{Lip-condition}
  \|F(\bar{U})-F(\tilde{U})\|_X&=&\|F_1(\bar{U})-F_1(\tilde{U})\|_{H^s}+ \|F_2(\bar{U})-F_2(\tilde{U})\|_{H^{s+1}}\nonumber\\
&= &  4\left\|\bar{U}_1\bar{U}_{2z}-\tilde{U}_1\tilde{U}_{2z}\right\|_{H^s}+\left\|\mathcal{K}\left(\left(\bar{U}_1-\tilde{U}_1\right)_z\right)
\right\|_{H^{s+1}}\nonumber\\
&\le& 4C_s\left\|\bar{U}_1\right\|_{H^s}\left\|\left(\bar{U}_{2}-\tilde{U}_{2}\right)_z\right\|_{H^s}+
4C_s\left\|\tilde{U}_{2z}\right\|_{H^s}\left\|\bar{U}_1-\tilde{U}_1\right\|_{H^s}+C_s\left\|\left(\bar{U}_1-\tilde{U}_1\right)_z\right\|_{H^{s-1}}\nonumber\\
&\le& 4C_s\left\|\bar{U}_1\right\|_{H^s}\left\|\bar{U}_{2}-\tilde{U}_{2}\right\|_{H^{s+1}}+
4C_s\left\|\tilde{U}_{2}\right\|_{H^{s+1}}\left\|\bar{U}_1-\tilde{U}_1\right\|_{H^s}+C_s\left\|\bar{U}_1-\tilde{U}_1\right\|_{H^{s}}\nonumber\\
&\le &\left(4C_sM+C_s\right)\left(\left\|\bar{U}_1-\tilde{U}_1\right\|_{H^s}+\left\|\bar{U}_2-\tilde{U}_2\right\|_{H^{s+1}}\right)\nonumber\\
&=&  C_s(4M+1) \left\|\bar{U}-\tilde{U}\right\|_X \;,
\end{eqnarray}
which proves that $F$ is locally Lipschitz continuous.

Now we can apply the Picard theorem on a Banach space to conclude that
there exists a time $T\left(\|u_0\|_{H^s(\Omega)},\|\psi_{0}\|_{H^{s+1}(\Omega)}\right)>0$ such that the system
\begin{eqnarray*}
U_t = F(U), \quad\quad U|_{t=0}=U_0\in O,
\end{eqnarray*}
has a unique solution $U=(v, \psi)\in
C^1\left([0,T),H^s(\Omega)\times V^{s+1}(\Omega)\right)$.
\end{proof}

\subsection{Bounded energy for the 3D model with mixed boundary conditions}
\begin{proposition}
\label{Prop-BC}
Let $\Omega = \Ox\times (0,b)$ and
$\Gamma =\{ (\Bx,z) \; | \; \Bx \in \Ox, \;z=0\}$.
Assume
$u_0|_{z=0} = u_0|_{z=b}=0$, $u_0\in H^2(\Omega)$,
$\psi_{0}\in H^3(\Omega)$ and satisfies \myref{psi-bc-1}.
Moreover, we assume that $\psi$ satisfies the following mixed
Dirichlet Robin boundary condition:
\begin{equation}
\label{psi-bc-1}
\psi|_{\partial \Omega \backslash \Gamma} = 0, \quad (\psi_z + \beta \psi )|_\Gamma = 0.
\end{equation}
Let $T$ be the largest time up to which the 3D inviscid model
\myref{model-u-mixed}-\myref{model-psi-mixed} has a smooth
solution with $u(t)\in H^2(\Omega)$ and
$\psi(t)\in H^3(\Omega)$ for $0 \leq t < T$. Then the following
identity holds
\begin{eqnarray}
\label{energy eqa}
\frac{d}{dt}\left(\int_\Omega \left(u^2+2|\nabla \psi|^2\right)
d\Bx dz -2\beta\int_0^a \int_0^a\psi^2|_{z=0}d\Bx \right)=0, \quad
0 \leq t < T.
\end{eqnarray}
Moreover, we have for $0 \leq t < T$ that
\begin{eqnarray}
\label{energy estimate}
\int_\Omega \left(u^2+2(1-\beta b)|\nabla \psi|^2\right) d \Bx dz \le
\int_\Omega \left(u_0^2+2|\nabla \psi_0|^2\right) d\Bx dz -
2\beta\int_0^a \int_0^a\psi_0^2|_{z=0} d \Bx .
\end{eqnarray}
\end{proposition}

\begin{remark}
One immediate consequence of the above proposition is
that if $\beta <1/b$, both $\int_\Omega u^2 d\Bx dz $ and
$\int_\Omega |\nabla \psi|^2 d \Bx dz $ are bounded.
\end{remark}

\begin{proof}
First of all, we know by the local existence result in
Theorem \ref{local wellposedness}
that there exists a $T_0$ such that the 3D inviscid model
\myref{model-u-mixed}-\myref{model-psi-mixed} has a unique
smooth solution with $u(t) \in H^2(\Omega)$ and $\psi(t) \in H^3(\Omega)$
for $0 \le t < T_0$.
Let $T$ be the largest time up to which the 3D inviscid model
\myref{model-u-mixed}-\myref{model-psi-mixed} has a smooth
solution with $u(t)\in H^2(\Omega)$ and
$\psi(t)\in H^3(\Omega)$ for $0 \leq t < T$. In the following,
we will perform energy estimates for
\myref{model-u-mixed}-\myref{model-psi-mixed} for
$0 \le t < T$.

First, we multiply \myref{model-u-mixed} by $u$
and integrate over $\Omega$. We obtain
\begin{eqnarray}
\label{energy u}
\frac{\mathd}{\mathd t}\int_\Omega u^2 \mathd {\Bx}\mathd z=4\int_\Omega u^2\psi_z\mathd {\Bx}\mathd z .
\end{eqnarray}
Next, we multiply \myref{model-psi-mixed} by $\psi$ and integrate over
$\Omega$ to obtain
\begin{eqnarray}
-\int_\Omega \Delta \psi_t \psi \mathd {\Bx}\mathd z = \int_\Omega \left(u^2\right)_z\psi \mathd {\Bx}\mathd z .
\end{eqnarray}
Integrating by parts and using boundary condition \myref{model psi
mixed bc}, we have
\begin{eqnarray}
\label{energy psi}
\frac{\mathd}{\mathd t}\int_\Omega |\nabla \psi|^2 \mathd {\Bx}\mathd z +2\int_{\Ox} \psi_{zt}\psi|_{z=0}\mathd {\Bx} = -2\int_\Omega u^2\psi_z \mathd {\Bx}\mathd z .
\end{eqnarray}
Multiplying \myref{energy psi} by 2 and adding the
resulting equation to \myref{energy u} gives
\begin{eqnarray}
\frac{\mathd}{\mathd t}\int_\Omega (u^2+2|\nabla \psi|^2 )
\mathd {\Bx}\mathd z &=&-4\int_{\Ox} \psi_{zt}\psi|_{z=0}
\mathd {\Bx}\nonumber\\
&=&4\beta\int_{\Ox} \psi_{t}\psi|_{z=0}\mathd {\Bx}\nonumber\\
&=&2\beta\frac{\mathd}{\mathd t}\int_{\Ox} \psi^2|_{z=0}\mathd {\Bx} ,
\end{eqnarray}
which gives \myref{energy eqa}.
On the other hand, we have the following estimate
\begin{eqnarray}
\int_{\Ox} \psi^2|_{z=0}\mathd {\Bx}&=&\int_{\Ox}\left(\int_0^b \psi_z \mathd z\right)^2\mathd {\Bx}\nonumber\\
&\le&b\int_{\Ox}\int_0^b \psi_z^2 \mathd z\mathd {\Bx}\le b\int_\Omega |\nabla \psi|^2 \mathd {\Bx}\mathd z .
\end{eqnarray}
This implies that
\begin{eqnarray}
\label{estimate psi boundary}
\int_\Omega \left(u^2+2|\nabla \psi|^2\right) \mathd {\Bx}\mathd z-2\beta\int_{\Ox} \psi^2|_{z=0}\mathd {\Bx}
\ge \int_\Omega \left(u^2+2(1-\beta b)|\nabla \psi|^2\right) \mathd {\Bx}\mathd z .
\end{eqnarray}
Combining \myref{energy eqa} with \myref{estimate psi boundary}, we
obtain
\begin{eqnarray}
\int_\Omega \left(u^2+2(1-\beta b)|\nabla \psi|^2\right) \mathd {\Bx}\mathd z\le
\int_\Omega \left(u_0^2+2|\nabla \psi_0|^2\right) \mathd {\Bx}\mathd z-2\beta\int_{\Ox} \psi_0^2|_{z=0}\mathd {\Bx} .
\end{eqnarray}
This completes the proof of Proposition \ref{Prop-BC}.
\end{proof}

\section{Blow-up of the 3D inviscid model}

In this section, we will prove that the 3D model \myref{model psi
mixed}-\myref{model psi mixed bc} develops a finite time
singularity for a class of smooth initial data with finite energy.
The finite time blowup is proved in a semi-infinite and a bounded
domain with mixed Dirichlet-Robin boundary conditions.

\subsection{Blow-up in a semi-infinite domain}

First, we consider the initial boundary value problem \myref{model
psi mixed}-\myref{model psi initial} in a semi-infinite domain
with $\Omega = \Ox \times (0,\infty)$. The main result is stated
in the theorem below:
\begin{theorem}
\label{theorem infty mix}
Assume that $u_0\in H^2(\Omega)$,
$u_0|_{\p \Omega}=0$, $u_0|_{\Omega}>0$,
$\psi_0\in H^3(\Omega)$ and satisfies \myref{model psi mixed bc}.
Further we assume that
 $\D\beta>\frac{\sqrt{2}\,\pi}{a}$
and $\beta \in S_\infty$ as defined in Lemma \ref{reisz}.
Choose $\al=\frac{2\pi^2}{\beta a^2}$, and define
\begin{eqnarray}
\label{phi mixed infty}
\phi(\Bx,z)=e^{-\al z} \phi_1(\Bx), \;\;
\phi_1(\Bx) =\sin \frac{\pi x_1}{a}\sin \frac{\pi x_2}{a},\quad (\Bx,z)\in
\Omega,
\end{eqnarray}
\begin{eqnarray*}
A=\int_\Omega (\log u_0)\phi \mathd {\Bx}\mathd z,\quad B=2\int_\Omega \psi_{0z}\phi \mathd {\Bx}\mathd z,\quad
D=\frac{\pi\al^{5/2}}{a(2\left(\frac{\pi}{a}\right)^2-\al^2)},\quad I_\infty=\int_0^\infty
\frac{dx}{\sqrt{x^3+1}}.
\end{eqnarray*}
If $A>0$ and $B>0$, then the 3D inviscid model \myref{model psi
mixed}, with the boundary condition \myref{model psi mixed bc} and
the initial data \myref{model psi initial}
 will develop a finite
time singularity in the $H^2$-norm no later than
\begin{eqnarray*}
T^*=\left(\frac{2DB}{3}{\frac{\sqrt{\al}\, \pi}{2a}}\right)^{-1/3}I_\infty.
\end{eqnarray*}
\end{theorem}

\begin{proof}
By Theorem \ref{local wellposedness}, we know that there exists a
finite time $T>0$ such that the system \myref{model psi
mixed}-\myref{model psi initial} has a unique smooth solution with
$u\in C^1([0,T),H^2(\Omega))$ and $\psi\in
C^1([0,T),H^3(\Omega))$. Let $T_b$ be the largest time such that
the system \myref{model psi mixed}-\myref{model psi mixed bc} with
initial condition $u_0, \psi_0$ has a smooth solution with $u \in
C^1([0,T_b);H^2(\Omega))$ and $\psi \in C^1([0,T_b);H^3(\Omega))$.
We claim that $T_b<\infty$. We prove this by contradiction.

Suppose that $T_b=\infty$, this means that for the given initial
data $u_0, \psi_0$, the system \myref{model psi
mixed}-\myref{model psi initial} has a globally smooth solution $
u \in C^1([0,\infty);H^2(\Omega))$ and $ \psi \in
C^1([0,\infty);H^3(\Omega))$. Multiplying $\phi_z$ to the both
sides of \myref{model-psi-mixed} and integrating over $\Omega$, we
get
\begin{eqnarray}
-\int_\Omega \Delta \psi_t \phi_z \mathd {\Bx}\mathd z=\int_\Omega \left(u^2\right)_{z}\phi_z \mathd {\Bx}\mathd z .
\end{eqnarray}
Note that $u|_{\p\Omega} = 0$ as long as the solution remains smooth.
By integrating by parts and using the boundary condition on $\psi$
and the property of $\phi$
to eliminate the boundary terms, we have
\begin{eqnarray}
-\int_\Omega \psi_{zt} \Delta \phi \mathd {\Bx}\mathd z-\int_{\Ox} \psi_{zt}\phi_z|_{z=0}\mathd {\Bx}-\int_{\Ox} \psi_t \Dx\phi|_{z=0}\mathd {\Bx}
 =\int_\Omega u^2 \phi_{zz}\mathd {\Bx}\mathd z .
\end{eqnarray}
Substituting $\phi$ into the above equation, we obtain
\begin{eqnarray}
\left(\frac{2\pi^2}{a^2}-\al^2\right)\frac{\mathd}{\mathd t}\int_\Omega \psi_z \phi \mathd {\Bx}\mathd z &= \al^2\int_\Omega u^2 \phi \mathd {\Bx}\mathd z-\int_{\Ox}\left.\left(\al\psi_{zt}+
\frac{2\pi^2}{a^2}\psi_t\right)\right|_{z=0}\phi_1(\Bx)
\mathd {\Bx}\nonumber\\
\label{psit-est1}
&= \al^2\int_\Omega u^2 \phi \mathd {\Bx}\mathd z+\int_{\Ox}\left(
\al\beta-\frac{2\pi^2}{a^2}\right)\psi_t|_{z=0}
\phi_1(\Bx) \mathd {\Bx} .
\end{eqnarray}
By the definition of $\al$, we have
\begin{eqnarray}
\al\beta-\frac{2\pi^2}{a^2}=0, \quad \mbox{and} \;\;
\al=\frac{2\pi^2}{\beta a^2}<\frac{\sqrt{2}\,\pi}{a} ,
\end{eqnarray}
since $\beta > \frac{\sqrt{2}\,\pi}{a}$. Thus the boundary
term on the right hand side of \myref{psit-est1} vanishes. We get
\begin{eqnarray}
\label{eq psi mixed infty} \frac{\mathd}{\mathd t}\int_\Omega \psi_z \phi \mathd {\Bx}\mathd z
&=& \frac{\al^2}{2\left(\frac{\pi}{a}\right)^2-\al^2}\int_\Omega u^2
\phi \mathd {\Bx}\mathd z .
\end{eqnarray}

Next, we multiply $\phi$ to \myref{model-logu} and integrate
over $\Omega$. We obtain
\begin{eqnarray}
\label{eq u mixed infty} \frac{\mathd}{\mathd t}\int_\Omega (\log u) \phi
\mathd {\Bx}\mathd z&=&2\int_\Omega \psi_z\phi \mathd {\Bx}\mathd z .
\end{eqnarray}
Combining \myref{eq psi mixed infty} with \myref{eq u mixed infty}, we
have
\begin{eqnarray}
\frac{\mathd^2}{\mathd t^2}\int_\Omega (\log u) \phi
\mathd {\Bx}\mathd z&=&\frac{2\al^2}{2\left(\frac{\pi}{a}\right)^2-\al^2}\int_\Omega
u^2 \phi \mathd {\Bx}\mathd z .
\end{eqnarray}
Integrating the above equation twice in time, we get
\begin{eqnarray}
\int_\Omega (\log u) \phi \mathd {\Bx}\mathd z &=& \frac{2\al^2}{2\left(\frac{\pi}{a}\right)^2-\al^2}\int_0^t\int_0^s\left(\int_\Omega u^2 \phi \mathd {\Bx}\mathd z\right)d\tau ds+A+Bt\nonumber\\
&\ge&\frac{2\al^2}{2\left(\frac{\pi}{a}\right)^2-\al^2}\int_0^t\int_0^s\left(\int_\Omega u^2 \phi \mathd {\Bx}\mathd z\right)d\tau ds+Bt .
\label{est-logu-1}
\end{eqnarray}
Note that $u > 0 $ for $(\Bx,z) \in \Omega$ and $t < T_b$.
It is easy to show that
\begin{eqnarray}
\int_\Omega (\log u) \phi \mathd {\Bx}\mathd z&\leq&
\int_\Omega (\log u)^+ \phi \mathd {\Bx}\mathd z
\leq \int_\Omega u\phi \mathd {\Bx}\mathd z\nonumber\\
&\leq&
\left(\int_\Omega \phi \mathd {\Bx}\mathd z\right)^{1/2}\left(\int_\Omega u^2\phi \mathd {\Bx}\mathd z\right)^{1/2}\nonumber\\
&=& \frac{2a}{\sqrt{\al}\,\pi}\left(\int_\Omega u^2\phi \mathd {\Bx}\mathd z\right)^{1/2},
\label{est-logu-2}
\end{eqnarray}
where $(\log u)^+ =\max(\log u, 0)$.
Combining \myref{est-logu-1} with \myref{est-logu-2} gives us the
crucial nonlinear dynamic estimate:
\begin{eqnarray}
\left(\int_\Omega u^2\phi \mathd {\Bx}\mathd z\right)^{1/2} \ge \frac{2\al^2}{2\left(\frac{\pi}{a}\right)^2-\al^2}\frac{\sqrt{\al}\,\pi}{2a}\int_0^t\int_0^s\left(\int_\Omega u^2 \phi \mathd {\Bx}\mathd z\right)d\tau ds
+\frac{\sqrt{\al}\,\pi}{2a}Bt.
\label{est-key}
\end{eqnarray}
Define
\begin{eqnarray}
\label{F-def}
F(t)=\frac{\pi\al^{5/2}}{a(2\left(\frac{\pi}{a}\right)^2-\al^2)}
\int_0^t\int_0^s\left(\int_\Omega u^2 \phi \mathd {\Bx}\mathd z\right)d\tau ds
+\frac{\sqrt{\al}\,\pi}{2a}Bt.
\end{eqnarray}
Then we have $\D F(0)=0$ and
$ \D F_t(0)=\frac{\sqrt{\al}\,\pi}{2a}B> 0$.
By differentiating \myref{F-def} twice in time and substituting
the resulting equation into \myref{est-key}, we obtain
\begin{eqnarray}
\label{eq F psi}
\frac{d^2F}{dt^2}&=&
\frac{\pi\al^{5/2}}{a(2\left(\frac{\pi}{a}\right)^2-\al^2)}
\int_\Omega u^2 \phi \mathd {\Bx}\mathd z\geq DF^2,
\end{eqnarray}
where $\D D=\frac{\pi\al^{5/2}}{a(2\left(\frac{\pi}{a}\right)^2-\al^2)}$.
Note that
$F_t=D\int_0^t\left(\int_\Omega u^2 \phi \mathd {\Bx}\mathd z\right) ds
+\frac{\sqrt{\al}\,\pi}{2a}B >0$. Multiplying
$F_t$ to \myref{eq F psi} and integrating in time, we get
\begin{eqnarray}
\label{dF psi}
\frac{dF}{dt}\geq \sqrt{\frac{2D}{3}F^3+C},
\end{eqnarray}
where $\D C= (F_t(0))^2 = \frac{\al\pi^2}{4a^2}B^2$.
Define
\begin{eqnarray*}
I(x)=\int_0^{x} \frac{dy}{\sqrt{y^3+1}},\quad J=\left(\frac{3C}{2D}\right)^{1/3} .
\end{eqnarray*}
Then, integrating \myref{dF psi} in time gives
\begin{eqnarray}
\label{blowup inequa}
I\left(\frac{F(t)}{J}\right)\geq \frac{\sqrt{C}t}{J},\quad \forall t \in \left[0,T^*\right] .
\end{eqnarray}
Note that both $I$ and $F$ are strictly increasing functions, and
$I(x)$ is uniformly bounded for all $x>0$ while the right hand side
increases linearly in time. It follows from \myref{blowup inequa}
that $F(t)$ must blow up no later than
\begin{eqnarray*}
\frac{J}{\sqrt{C}}I_\infty=T^*.
\end{eqnarray*}
This contradicts with the assumption that the 3D model has a
globally smooth solution. This contradiction implies that the
solution of the system \myref{model psi mixed} must develop a
finite time singularity no later than $T^*$.
\end{proof}

\begin{remark}
As we can see in the proof of Theorem \ref{theorem infty mix},
the same conclusion still holds if we replace the boundary
condition
\[
(\psi_z + \beta \psi)|_{z=0} = 0,
\]
by the following integral constraint
\begin{eqnarray*}
\int_{\Ox}\left.\left(\psi_z+
\beta \psi\right)\right|_{z=0}\sin \frac{\pi x_1}{a}\sin \frac{\pi x_2}{a}\mathd {\Bx}=0.
\end{eqnarray*}
%So there exist a class of boundary conditions, the finite time blowup can be guaranteed.
%One example is that
%\begin{eqnarray*}
%\psi_t(x,0,t)=-\gamma \sin\frac{\pi x}{a} \int_0^\infty\int_0^\pi u^2\sin\frac{\pi x}{a} e^{-\pi z/a}\mathd {\Bx}\mathd z,\quad \gamma>0.
%\end{eqnarray*}
\end{remark}

\subsection{Blow-up in a bounded domain}

In this subsection, we will prove finite time blow-up of the 3D
model in a  bounded domain. First, we formulate the initial
boundary problem of the 3D model as follows:
\begin{eqnarray}
\label{model psi bounded}
&&\left\{\begin{array}{rcl}
u_t &=&2u\psi_z\\
-\Delta \psi_t&=&\left(u^2\right)_z
\end{array}\right., \quad (\Bx,z)\in \Omega=\Ox\times(0,b),\\
\label{model psi bounded bc}
&&\psi|_{\p\Omega\backslash\Gamma}=0,\quad
\left(\psi_z+\beta \psi\right)|_\Gamma=0,\\
&& \psi|_{t=0}=\psi_0(\Bx,z),\quad u|_{t=0}=u_0(\Bx,z)\ge0,\nonumber
\end{eqnarray}
where $\Bx=(x_1,x_2),\;\Ox=(0,a)\times(0,a),\;
\Gamma=\left\{(\Bx,z)\in \Omega \; |\; \Bx \in \Ox, \;z=0 \right\}$.
We can get a similar blow-up result which is summarized below:
\begin{theorem}
\label{theorem bounded mix}
Assume that $u_0\in H^2(\Omega)$, $u_0|_{\p \Omega}=0$, $u_0|_{\Omega}>0$,
$\psi_0\in H^3(\Omega)$ and satisfies \myref{model psi bounded bc}.
Further, we assume that
$\beta \in S_b$ as defined in Lemma \ref{reisz} and satisfies
$\D \beta>\frac{\sqrt{2}\pi}{a}
\left(\frac{e^{\sqrt{2}\pi b/a}+e^{-\sqrt{2}\pi b/a}}
{e^{\sqrt{2}\pi b/a}-e^{-\sqrt{2}\pi b/a}}\right)$. Define
\begin{eqnarray}
\label{phi mixed bounded}
\phi(\Bx,z)=\frac{e^{-\alpha (z-b)}+e^{\alpha (z-b)}}{2}\sin \frac{\pi
x_1}{a}\sin \frac{\pi x_2}{a},\quad (\Bx,z)\in \Omega,
\end{eqnarray}
where $\al$ satisfies $0<\al<\sqrt{2}\pi/a$ and
$ 2\left(\frac{\pi}{a}\right)^2
\frac{e^{\al b}-e^{-\al b}}{\al (e^{\al b}+e^{-\al b})} = \beta$.
Let
\begin{eqnarray*}
A=\int_\Omega (\log u_0)\phi \mathd {\Bx}\mathd z,\quad B=2\int_\Omega \psi_{0z}\phi \mathd {\Bx}\mathd z,\quad
D=\frac{\pi\al^{5/2}}{a(2\left(\frac{\pi}{a}\right)^2-\al^2)}
,\quad I_\infty=\int_0^\infty \frac{\mathd {\Bx}}{\sqrt{x^3+1}}.
\end{eqnarray*}
If $A>0$ and $B>0$,
then the solution of
\myref{model psi bounded}-\myref{model psi bounded bc}
will blow up no later than
\begin{eqnarray*}
T^*=\left(\frac{2DB}{3}\frac{\sqrt{\al} \pi}{2a}\right)^{-1/3}I_\infty.
\end{eqnarray*}
\end{theorem}

\begin{proof}
We follow the same strategy as in the proof of Theorem
\ref{theorem infty mix}. By Theorem \ref{local wellposedness}, we
know that there exists a finite time $T>0$ such that the system
\myref{model psi bounded} has a unique smooth solution with $u\in
C^1([0,T),H^2(\Omega))$ and $\psi\in C^1([0,T),H^3(\Omega))$ for
$0\le t<T$. Let $T_b$ be the largest time time such that the
system \myref{model psi bounded}-\myref{model psi bounded bc} with
initial condition $u_0, \psi_0$ has a smooth solution with $u \in
C^1([0,T_b);H^2(\Omega))$ and $\psi \in C^1([0,T_b);H^3(\Omega))$.
We claim that $T_b<\infty$. We prove this by contradiction.

Suppose that $T_b=\infty$, this means that for the given
initial data $u_0, \psi_0$, the system \myref{model psi bounded}
has a globally smooth solution with
$u \in C^1([0,\infty);H^2(\Omega))$ and
$\psi \in C^1([0,\infty);H^3(\Omega))$.
Multiplying $\phi_z$ to the both sides of the $\psi$-equation
and integrating over $\Omega$, we get
\begin{eqnarray}
-\int_\Omega \Delta \psi_t \phi_z \mathd {\Bx}\mathd z=\int_\Omega \left(u^2\right)_{z}\phi_z \mathd {\Bx}\mathd z .
\end{eqnarray}
Note that $u|_{\p \Omega} = 0$ as long as the solution remains smooth.
By integrating by parts and using the boundary condition on $\psi$
and the property of $\phi$ to eliminate the boundary terms, we have
\begin{eqnarray}
-\int_\Omega \psi_{zt} \Delta \phi \mathd {\Bx}\mathd z-\int_{\Ox} \psi_{zt}\phi_z|_{z=0}\mathd {\Bx}-\int_{\Ox} \psi_t \Dx\phi|_{z=0}\mathd {\Bx}
 =\int_\Omega u^2 \phi_{zz}\mathd {\Bx}\mathd z .
\label{integral-psi-1}
\end{eqnarray}
Substituting $\phi$ to \myref{integral-psi-1}
and using the boundary condition for $\psi$, we obtain
\begin{eqnarray}
&&\left(2\frac{\pi^2}{a^2}-\al^2\right)\frac{\mathd}{\mathd t}\int_\Omega \psi_z \phi \mathd {\Bx}\mathd z\nonumber\\
&=& \al^2\int_\Omega u^2 \phi \mathd {\Bx}\mathd z-\int_{\Ox}\left.\left(\frac{\al}{2}\left(e^{\al b}-e^{-\al b}\right)\psi_{zt}+
\frac{\pi^2}{a^2}\left(e^{\al b}+e^{-\al b}\right)\psi_t\right)\right|_{z=0}\phi_1(\Bx)\mathd {\Bx}\nonumber\\
&=& \al^2\int_\Omega u^2 \phi \mathd {\Bx}\mathd z+\int_{\Ox}\left(
\frac{\al}{2}\left(e^{\al b}-e^{-\al b}\right)\beta-\frac{\pi^2}{a^2}\left(e^{\al b}+e^{-\al b}\right)\right)\psi_t|_{z=0}\phi_1(\Bx)
\mathd {\Bx}\nonumber\\
\label{ori eqa mix periodic}
&=&\al^2\int_\Omega u^2 \phi \mathd {\Bx}\mathd z+\frac{\al}{2}\left(e^{\al b}-e^{-\al b}\right)\int_{\Ox}\left(
\beta-2\left(\frac{\pi}{a}\right)^2\frac{e^{\al b}+e^{-\al b}}{\al\left(e^{\al b}-e^{-\al b}\right)
}\right)\psi_t|_{z=0}\phi_1(\Bx) \mathd {\Bx} ,\quad\quad\quad
\end{eqnarray}
where $\phi_1(\Bx) = \sin \frac{\pi x_1}{a}\sin \frac{\pi x_2}{a}$.
Let
$\D h(\al)=2\left(\frac{\pi}{a}\right)^2
\frac{e^{\al b}+e^{-\al b}}{\al\left(e^{\al b}-e^{-\al b}\right)}$.
Direct computations show that
$\frac{d}{d \al} h(\al) < 0$ for all $\al > 0$. Thus we have
\begin{eqnarray}
\frac{\sqrt{2}\pi}{a}
\left(\frac{e^{\sqrt{2} \pi b/a}+e^{-\sqrt{2} \pi b/a}}
{e^{\sqrt{2} \pi b/a}-e^{-\sqrt{2} \pi b/a}} \right )
= h\left (\frac{\sqrt{2}\pi}{a}\right )
< h(\al) < h(0_+) = \infty, \quad
0 <\al<\frac{\sqrt{2}\pi}{a}.
\end{eqnarray}
Since $\beta > h\left (\frac{\sqrt{2}\pi}{a}\right )$ by
assumption, we can choose a unique $\al$ with
$0<\al<\frac{\sqrt{2}\pi}{a}$ such that
\begin{equation}
\label{alpha-beta}
\frac{2\pi^2}{a^2}
\frac{e^{\al b}+e^{-\al b}}
{\al(e^{\al b}-e^{-\al b})} = \beta.
\end{equation}
With this choice of $\al$, the boundary term in
\myref{ori eqa mix periodic} vanishes. Therefore we obtain
\begin{eqnarray}
\label{eq psi mixed periodic}
\frac{\mathd}{\mathd t}\int_\Omega \psi_z \phi \mathd {\Bx}\mathd z
&=& \frac{\al^2}{2\left(\frac{\pi}{a}\right)^2-\al^2}\int_\Omega u^2
\phi \mathd {\Bx}\mathd z .
\end{eqnarray}

Next, we multiply $\phi$ to \myref{model-logu} and integrate
over $\Omega$. We get
\begin{eqnarray}
\label{eq u mixed infty-1}
\frac{\mathd}{\mathd t}\int_\Omega (\log u) \phi
\mathd {\Bx}\mathd z&=&2\int_\Omega \psi_z\phi \mathd {\Bx}\mathd z .
\end{eqnarray}
Combining \myref{eq u mixed infty-1} with
\myref{eq psi mixed periodic}, we get
\begin{eqnarray}
\frac{\mathd^2}{\mathd t^2}\int_\Omega (\log u) \phi
\mathd {\Bx}\mathd z&=&\frac{2\al^2}{2\left(\frac{\pi}{a}\right)^2-\al^2}\int_\Omega
u^2 \phi \mathd {\Bx}\mathd z .
\end{eqnarray}
Now we can follow the exactly same procedure as in the proof of
Theorem \ref{theorem infty mix} to prove that the 3D model
must develop a finite time blow-up.
\end{proof}

\begin{remark}
We remark that the same conclusion is still true if we
replace the Dirichlet boundary condition $\psi|_{z=b} =0$
by the Neumann boundary condition $\psi_z|_{z=b}=0$.
The only difference is that the weight function $\phi$
is now changed to
\begin{eqnarray}
\label{phi mixed bounded neumann}
\phi(\Bx,z)=\frac{e^{-\alpha (z-b)}-e^{\alpha (z-b)}}{2}\sin \frac{\pi
x_1}{a}\sin \frac{\pi x_2}{a},\quad (\Bx,z)\in \Omega,
\end{eqnarray}
where $0<\al<\sqrt{2}\pi/a$, and
$\beta$ satisfies a variant of \myref{set-S} in Lemma \ref{reisz} and
\begin{eqnarray}
\D \frac{\sqrt{2}\pi}{a}
\left(\frac{e^{\sqrt{2}\pi b/a}-e^{-\sqrt{2}\pi b/a}}
{e^{\sqrt{2}\pi b/a}+e^{-\sqrt{2}\pi b/a}} \right )
< \beta <2b\left(\frac{\pi}{a}\right)^2.
\end{eqnarray}
We omit the proof here.
\end{remark}

\subsection{Blow-up of a generalized 3D model}

In this section, we study singularity formation of a generalized
3D model by changing the sign of the Laplace operator in
the $\psi$-equation \myref{model-psi-mixed}.
Specifically, we consider the following
generalized 3D model:
\begin{eqnarray}
\label{model psi opposite}
&&\left\{\begin{array}{rcl}
u_t &=&2u \psi_z\\
\Delta \psi_t&=&\left(u^2\right)_z
\end{array}\right., \quad (\Bx,z)\in \Omega=\Ox\times(0,a)=(0,a)\times(0,a)\times(0,a).
\end{eqnarray}
The boundary and initial conditions are below
\begin{eqnarray}
\label{BC-opposite}
&&\psi|_{\p\Omega\backslash \Gamma}=0,\quad \psi_z|_{\Gamma}=0,
\quad \Gamma =\{ (\Bx,z) \:|\: \Bx \in \Ox, \; z=0, \; \mbox{or}\; z=a\}
\\
&& \psi|_{t=0}=\psi_0(\Bx,z),\quad u|_{t=0}=u_0(\Bx,z)\ge0. \nonumber
\end{eqnarray}
In this subsection, we will generalize the singularity analysis
presented in the previous subsection to prove
that the solution of the generalized
3D model will develop a finite time singularity.
The main result is summarized in the following theorem.
\begin{theorem}
\label{theorem opposite 0}
Assume that $u_0\in H^2(\Omega)$,
$u_0|_{\p\Omega}=0$, $u_0|_{\Omega}>0$,
$\psi_0\in H^3(\Omega)$ and satisfies \myref{BC-opposite}.
Further, we define
\begin{eqnarray}
\label{phi opposite}
\phi(\Bx,z)=\sin \frac{\pi
x_1}{a}\sin \frac{\pi x_2}{a}\sin \frac{\pi z}{a},\quad (\Bx,z)\in \Omega.
\end{eqnarray}
Let
\begin{eqnarray*}
A=\int_\Omega (\log u_0)\phi \mathd {\Bx}\mathd z,\quad B=2\int_\Omega \psi_{0z}\phi \mathd {\Bx}\mathd z,\quad
I_\infty=\int_0^\infty \frac{\mathd {\Bx}}{\sqrt{x^3+1}} .
\end{eqnarray*}
If $A>0$ and $B>0$,
then the solution of \myref{model psi opposite}-\myref{BC-opposite}
will blow up
no later than $T^*=\left(\frac{B}{18}\right)^{-1/3}I_\infty$.
\end{theorem}
\begin{proof}
First, by using an argument similar to the local well-posedness
result in Theorem \ref{local wellposedness}, we can prove that the
system \myref{model psi opposite}-\myref{BC-opposite} is locally well-posed.
We prove the theorem by contradiction.
Suppose that the system \myref{model psi opposite}-\myref{BC-opposite}
has a globally smooth solution with
$u\in C^1([0,\infty);H^2(\Omega))$ and
$\psi \in C^1([0,\infty);H^3(\Omega))$.
Multiplying $\phi_z$ to the both sides of the $\psi$-equation
and integrating over $\Omega$, we get
\begin{eqnarray}
-\int_\Omega \Delta \psi_t \phi_z \mathd {\Bx}\mathd z=\int_\Omega \left(u^2\right)_{z}\phi_z \mathd {\Bx}\mathd z .
\end{eqnarray}
Note that $u|_{\p\Omega} = 0$ as long as the solution remains smooth.
By integrating by parts and using the boundary condition on $\psi$
and the property of $\phi$ to eliminate the boundary terms, we have
\begin{eqnarray}
-\int_\Omega \psi_{zt} \Delta \phi \mathd {\Bx}\mathd z=\int_\Omega u^2 \phi_{zz}\mathd {\Bx}\mathd z .
\label{integral-psi-1 opposite}
\end{eqnarray}
Substituting $\phi$ to \myref{integral-psi-1 opposite}
and using the boundary condition for $\psi$, we obtain
\begin{eqnarray}
\label{eqn psiz opposite}
\frac{\mathd}{\mathd t}\int_\Omega \psi_z \phi \mathd {\Bx}\mathd z= \frac{1}{3}\int_\Omega u^2 \phi \mathd {\Bx}\mathd z .
\end{eqnarray}
Next, we multiply $\phi$ to \myref{model-logu} and integrate
over $\Omega$. We obtain
\begin{eqnarray}
\label{log u opposite}
\frac{\mathd}{\mathd t}\int_\Omega (\log u) \phi \mathd {\Bx}\mathd
z=2\int_\Omega \psi_z\phi \mathd {\Bx}\mathd z .
\end{eqnarray}
Combining \myref{eqn psiz opposite} with \myref{log u opposite}, we obtain
\begin{eqnarray}
\frac{\mathd^2}{\mathd t^2}\int_\Omega (\log u) \phi \mathd {\Bx}\mathd z=\frac{2}{3}\int_\Omega u^2 \phi \mathd {\Bx}\mathd z .
\end{eqnarray}
Integrating the above equation twice in time, we get
\begin{eqnarray}
\label{final ineq opposite} \int_\Omega (\log u) \phi \mathd
{\Bx}\mathd z &=& \frac{2}{3}\int_0^t\int_0^s\left(\int_\Omega u^2
\phi \mathd {\Bx}\mathd z\right)d\tau ds+A+Bt .
\end{eqnarray}
Using \myref{final ineq opposite}, following the same argument as in
Theorem \ref{theorem infty mix}, we can prove that the solution of the initial boundary value problem \myref{model psi opposite}-\myref{BC-opposite}
blows up no later than $T^*$.
\end{proof}

\section{Blow-up of the 3D model with partial viscosity}

In this section, we prove finite blow-up of the 3D model
with partial viscosity. Specifically, we consider the
following initial boundary value problem in a semi-infinite
domain:
\begin{eqnarray}
\label{model psi mixed vis} &&\left\{\begin{array}{rcl}
u_t &=&2u\psi_z\\
\omega_t&=&\left(u^2\right)_z+\nu \Delta \omega\\
-\Delta\psi&=&\omega .
\end{array}\right., \quad (\Bx,z)\in \Omega=\Ox\times(0,\infty),
\end{eqnarray}
The initial and boundary conditions are given as follows:
\begin{eqnarray}
\label{BC-psi-vis}
&&\psi|_{\p\Omega\backslash\Gamma}=0,\quad \left(\psi_z+\beta \psi\right)|_\Gamma=0, \\
\label{BC-w-vis}
&&\omega|_{\p\Omega\backslash\Gamma}=0,\quad \left(\omega_z+\gamma \omega\right)|_\Gamma=0, \\
&& \omega|_{t=0}=\omega_0(\Bx,z),\quad u|_{t=0}=u_0(\Bx,z)\ge0 ,
\label{IC-psi-w2}
\end{eqnarray}
where $\Gamma=\left\{(\Bx,z)\in \Omega \;|\;\Bx \in \Ox, \; z=0\right\}$.

Now we state the main result of this section.
\begin{theorem}
\label{theorem vis blowup mixed}
Assume that $u_0|_{\p\Omega}=0$, $u_{0z}|_{\p\Omega}=0$, 
 $u_0|_{\Omega}>0$,
$u_0 \in H^2(\Omega)$, $\psi_0 \in H^3(\Omega)$,
$\omega_0 \in H^1(\Omega)$, $\psi_0$ satisfies
\myref{BC-psi-vis} and $\omega_0$ satisfies \myref{BC-w-vis}.
Further, we assume that
$\beta \in S_\infty$ as defined in Lemma \ref{reisz} and
$\beta >\frac{\sqrt{2}\pi}{a}$, $\gamma =\frac{2\pi^2}{\beta a^2}$.
Let
\begin{eqnarray}
\label{phi vis infty}
\phi(\Bx,z)=e^{-\al z}\sin \frac{\pi x_1}{a}\sin \frac{\pi x_2}{a},\quad (\Bx,z)\in
\Omega,
\end{eqnarray}
where $\al = \frac{2 \pi^2}{\beta a^2}$ satisfies
$0<\al<\sqrt{2}\pi/a$.
Define
\begin{eqnarray}
&& A=\int_\Omega (\log u_0) \phi \mathd {\Bx}\mathd z,\quad B=-\int_\Omega \omega_0\phi_z
\mathd {\Bx}\mathd z, \quad
D=\frac{2}{2\left(\frac{\pi}{a}\right)^2-\al^2},\\
&&I_\infty=\int_0^\infty \frac{\mathd {\Bx}}{\sqrt{x^3+1}},\quad
T^*=\left(\frac{ \pi\al^3 D^2B}{12a}\right)^{-1/3}I_\infty .
\end{eqnarray}
If $A>0$, $B>0$, and $T^*<(\log
2)\left(\nu\left(\frac{2\pi^2}{a^2}-\al^2\right)\right)^{-1}$,
then the solution of model \myref{model psi mixed vis} with
initial and boundary conditions
\myref{BC-psi-vis}-\myref{IC-psi-w2}
 will develop a finite time singularity before $T^*$.
\end{theorem}
\begin{proof}
First of all, we can prove that the 3D model \myref{model psi mixed vis}
with initial and boundary conditions given by 
\myref{BC-psi-vis}-\myref{IC-psi-w2} 
has a unique solution, $u\in C([0,T],H^{2}(\Omega))$, 
$\omega\in C([0,T],H^{1}(\Omega))$ and
$\psi\in C([0,T],H^{3}(\Omega))$ for some $T>0$ depending on
initial data. There are two key ingredients in this analysis.
The first one is to design a Picard iteration for the 3D model.
The second one is to show that the mapping that generates the
Picard iteration is a contraction mapping and the Picard
iteration converges to a fixed point of the Picard mapping
by using the Contraction Mapping Theorem. To establish the contraction 
property of the Picard mapping, we need to use the well-posedness 
property of the heat equation with the same Dirichlet Robin 
boundary condition as $\omega$. The well-posedness analysis
of the heat equation with a mixed Dirichelet Robin boundary
has been studied in the literature. The case of $\gamma >0$ is
more subtle because there is a growing eigenmode. Since the
complete analysis of the local well-posedness of 3D model with 
partial viscosity is quite technical, we will not present the 
analysis here and refer the reader to \cite{HSW11} for the details 
of the analysis.
  
We are now ready to prove the finite time singularity of
the 3D model with partial viscosity with the given initial
boundary data. We will prove the theorem by contradiction. 
Assume that the 3D model \myref{model psi mixed vis} with 
initial and boundary conditions
\myref{BC-psi-vis}-\myref{IC-psi-w2} has a globally smooth
solution, $u \in C^1([0,\infty);H^2(\Omega))$, $\psi \in
C^1([0,\infty);H^3(\Omega))$, and $\omega \in
C^1([0,\infty);H^1(\Omega))$. Multiplying $\phi$ to the both sides
of the $\psi$-equation and integrating over $\Omega$, we get
\begin{eqnarray}
-\int_\Omega \Delta \psi \phi_z \mathd {\Bx}\mathd z=\int_\Omega \omega \phi_z \mathd {\Bx}\mathd z .
\end{eqnarray}
By integrating by parts and using boundary conditions
\myref{BC-psi-vis}-\myref{BC-w-vis} and the property of $\phi$,
we obtain
\begin{eqnarray}
\int_\Omega \psi_{z} \Delta \phi \mathd {\Bx}\mathd z-\int_{\Ox}
\psi_{z}\phi_z|_{z=0}\mathd {\Bx}\mathd z-\int_{\Ox} \psi \Dx \phi|_{z=0}\mathd {\Bx}\mathd z
=\int_\Omega \omega \phi_{z}\mathd {\Bx}\mathd z .
\end{eqnarray}
Substituting $\phi$ defined in \myref{phi vis infty} into the
above equation, we have
\begin{eqnarray}
-\left(\frac{2\pi^2}{a^2}-\al^2\right)\int_\Omega \psi_z \phi \mathd {\Bx}\mathd z
&=& \int_\Omega \omega \phi_{z}
\mathd {\Bx}\mathd z-\int_{\Ox}\left.\left(\al\psi_{z}+
\frac{2\pi^2}{a^2}\psi\right)\right|_{z=0}
\phi_1 (\Bx) \mathd {\Bx}\nonumber\\
&=& \int_\Omega \omega \phi_{z} \mathd {\Bx}\mathd z+\int_{\Ox}\left(
\al\beta-\frac{2\pi^2}{a^2}\right)\psi|_{z=0}
\phi_1(\Bx) \mathd {\Bx},
\label{eqn-psi-vis-bc}
\end{eqnarray}
where $\phi_1(\Bx) = \sin \frac{\pi x_1}{a}\sin \frac{\pi x_2}{a}$.
Since $\beta >\frac{\sqrt{2}\pi}{a}$, we can choose
\begin{eqnarray}
\al =\frac{2 \pi^2}{\beta a^2}<\frac{\sqrt{2}\pi}{a},
\end{eqnarray}
to eliminate the boundary term in \myref{eqn-psi-vis-bc}.
This gives rise to the following identity:
\begin{eqnarray}
\label{eqn-psiz-int}
\int_\Omega \psi_z \phi \mathd {\Bx}\mathd z &=&
-\frac{1}{2\left(\frac{\pi}{a}\right)^2-\al^2}\int_\Omega \omega
\phi_{z} \mathd {\Bx}\mathd z .
\end{eqnarray}
Next, we multiply $\phi_z$ to the both sides of the $\omega$-equation
and integrate over $\Omega$
\begin{eqnarray}
\int_\Omega \omega_t \phi_z \mathd {\Bx}\mathd z = \int_\Omega \left(u^2\right)_z
\phi_z \mathd {\Bx}\mathd z
+\nu\int_\Omega \Delta\omega \phi_z \mathd {\Bx}\mathd z .
\end{eqnarray}
Integrating by parts and using $u|_{\p\Omega} = 0$, we obtain
\begin{eqnarray}
\frac{\mathd}{\mathd t}\int_\Omega \omega \phi_z \mathd {\Bx}\mathd z&=&-\int_\Omega u^2 \phi_{zz}
\mathd {\Bx}\mathd z+ \nu \left(-\int_{\Ox} \omega_z \phi_z|_{z=0}\mathd {\Bx} +\int_{\Ox} \omega
\phi_{zz}|_{z=0}\mathd {\Bx}+\int_\Omega \omega \Delta \phi_z \mathd {\Bx}\mathd z
\right)\nonumber\\
&=&-\al^2\int_\Omega u^2 \phi \mathd {\Bx}\mathd z
-\nu\left(\frac{2\pi^2}{a^2}-\al^2\right)\int_\Omega \omega \phi_z
\mathd {\Bx}\mathd z+\nu\int_{\Ox} (\al \omega_z+\al^2\omega)|_{z=0}
\phi_1(\Bx)\mathd {\Bx}\nonumber\\
&=&-\al^2\int_\Omega u^2 \phi \mathd {\Bx}\mathd z
-\nu\left(\frac{2\pi^2}{a^2}-\al^2\right)\int_\Omega \omega \phi_z
\mathd {\Bx}\mathd z+\nu\int_{\Ox} \al(\al-\gamma)\omega|_{z=0}
\phi_1(\Bx) \mathd {\Bx}\nonumber\\
&=&-\al^2\int_\Omega u^2 \phi \mathd {\Bx}\mathd z
-\nu\left(\frac{2\pi^2}{a^2}-\al^2\right)\int_\Omega \omega \phi_z
\mathd {\Bx}\mathd z ,
\end{eqnarray}
where we have used $\al = \gamma$ to eliminate the boundary
term in the above estimates. Solving the above ordinary equation
for $\int_\Omega \omega \phi_z \mathd {\Bx}\mathd z $ gives
\begin{eqnarray}
\label{eqn-w-int}
\int_\Omega \omega \phi_z \mathd {\Bx}\mathd z=e^{-\lambda t}\int_\Omega
\omega_0\phi_z \mathd {\Bx}\mathd z -\al^2\int_0^t e^{-\lambda(t-s)}\left(\int_\Omega
u^2\phi \mathd {\Bx}\mathd z\right)ds ,
\end{eqnarray}
where $\lambda =\nu\left(\frac{2\pi^2}{a^2}-\al^2\right)$.
Using the reformulated $u$-equation
\myref{model-logu}, \myref{eqn-psiz-int} and \myref{eqn-w-int},
we obtain
\begin{eqnarray}
\frac{\mathd}{\mathd t}\int_\Omega (\log u) \phi \mathd {\Bx}\mathd z&=&2\int_\Omega \psi_z\phi
\mathd {\Bx}\mathd z\\
&=&\frac{2}{2\left(\frac{\pi}{a}\right)^2-\al^2}\left(-e^{-\lambda
t}\int_\Omega \omega_0\phi_z \mathd {\Bx}\mathd z +\al^2\int_0^t
e^{-\lambda(t-s)}\left(\int_\Omega u^2\phi \mathd {\Bx}\mathd z\right)ds\right). \nonumber
\end{eqnarray}
Integrating the above equation in time, we get
\begin{eqnarray}
\int_\Omega (\log u) \phi \mathd {\Bx}\mathd z&=&\int_\Omega (\log u_0) \phi
\mathd {\Bx}\mathd z-\frac{2}{2\left(\frac{\pi}{a}\right)^2-\al^2}
\left (\frac{1-e^{-\lambda t}}{\lambda}\right )
 \left (-\int_\Omega \omega_0\phi_z
\mathd {\Bx}\mathd z \right ) \nonumber \\
&& +\frac{2\al^2}{2\left(\frac{\pi}{a}\right)^2-\al^2}\int_0^t
\int_0^s e^{-\lambda(s-\tau)}\left(\int_\Omega u^2\phi
\mathd {\Bx}\mathd z\right)d\tau ds .
\label{logu-est-vis}
\end{eqnarray}
Let $\D T_0=\frac{\log 2}{\lambda}$, then $e^{-\lambda
t}\ge\frac{1}{2}$ over the interval $[0,T_0]$. Note that
$\frac{d}{dt}\left(\frac{1-e^{-\lambda t}}{\lambda}\right )
= e^{-\lambda t}\ge\frac{1}{2}$ for $ 0 \le t \le T_0$.
This implies that $\frac{1-e^{-\lambda t}}{\lambda} \ge \frac{t}{2} $
for $0 \le t \le T_0$. Thus we have from \myref{logu-est-vis} that
\begin{eqnarray}
\int_\Omega (\log u) \phi \mathd {\Bx}\mathd z&\ge&A+\frac{1}{2}DBt
+\frac{1}{2}D\al^2\int_0^t \int_0^s \left(\int_\Omega u^2\phi
\mathd {\Bx}\mathd z\right)d\tau ds,
\end{eqnarray}
for all $ t\in [0,T_0] $.
Now we can follow exactly the same procedure as in the proof of
Theorem \ref{theorem infty mix} to prove that the 3D model
must develop a finite time blow-up before
\begin{eqnarray}
T^*=\left(\frac{\al^3 \pi D^2B}{12a}\right)^{-1/3}I_\infty.
\end{eqnarray}
Since $T^*<T_0$, we conclude that the solution must blow up
before $T^*$.
\end{proof}

\begin{remark}
We can also prove the finite time blow-up of the 3D model with
partial viscosity in a bounded domain following a similar
argument. We omit the analysis here.
\end{remark}

\section{Blow-up of the 3D model with conservative boundary conditions}

In this section, we will consider boundary conditions for $\psi$
that will conserve energy. Under some additional condition,
we can prove that the solution of the 3D model with conservative
boundary conditions will also develop a finite time singularity.

\subsection{Blow-up in a semi-infinite domain}

Consider the following initial boundary value problem:
\begin{eqnarray}
\label{model psi conserved}
&&\left\{\begin{array}{rcl}
u_t &=&2u\psi_z\\
-\Delta \psi_t&=&\left(u^2\right)_z
\end{array}\right., \quad (x,z)\in \Omega=\Ox\times(0,\infty),\\
\label{model psi conserved bc}
&&\psi|_{\p\Omega\backslash\Gamma}=0,\quad \psi_z|_\Gamma=0, \\
&& \psi|_{t=0}=\psi_0(\Bx,z),\quad u|_{t=0}=u_0(\Bx,z)\ge0 ,\nonumber
\end{eqnarray}
where $\Ox=(0,a)\times(0,a)$, and
$\Gamma=\left\{(\Bx,z)\in \Omega \;|\;\Bx \in \Ox, \; z=0\right\}$.

\begin{theorem}
\label{theorem infty conserved}
Assume that $u_0\in H^2(\Omega)$, $u_0|_{\p \Omega}=0$ , $u_0|_{\Omega}>0$,
$\psi_0\in H^3(\Omega)$ and satisfies \myref{model psi conserved bc}.
Let
\begin{eqnarray}
\label{phi infty conserved}
\phi(\Bx,z)=e^{-\al z}\sin \frac{\pi x_1}{a}\sin \frac{\pi x_2}{a},\quad (x,z)\in \Omega,
\end{eqnarray}
with $\al = \frac{\pi}{a}$,
and
\begin{eqnarray*}
&&A=\int_\Omega (\log u_0)\phi \mathd {\Bx}\mathd z,\quad B=2\int_\Omega \psi_{0z}\phi \mathd {\Bx}\mathd z,\\
&&r(t)=\frac{4\left (\frac{\pi}{a}\right )^2}{2\left (\frac{\pi}{a}\right)^2-\al^2}\int_{\Ox}
(\psi-\psi_{0})|_{z=0}\sin \frac{\pi x_1}{a}\sin \frac{ \pi x_2}{a} \mathd {\Bx}.
\end{eqnarray*}
If $A>0,\;B>0$ and $r(t)\le \frac{B}{2}$ as long as $u,\psi$ remain regular,
then the solution of \myref{model psi conserved}-\myref{model psi conserved bc}
 will develop a finite time
singularity in the $H^2$ norm.
\end{theorem}

\begin{proof}
First, by using an argument similar to the local well-posedness
result in Theorem \ref{local wellposedness}, we can prove that the
system \myref{model psi conserved}-\myref{model psi conserved bc}
is locally well-posed.
We prove the theorem by contradiction. Assume that the initial boundary
value problem has a globally smooth solution with
$u\in C^1([0,\infty);H^2(\Omega))$ and
$\psi\in C^1([0,\infty);H^3(\Omega))$.
Multiplying $\phi_z$ to the both sides of the $\psi$-equation and
integrating over $\Omega$, we get
\begin{eqnarray}
-\int_\Omega \Delta \psi_t \phi_z \mathd {\Bx}\mathd z=\int_\Omega \left(u^2\right)_{z}\phi_z \mathd {\Bx}\mathd z .
\end{eqnarray}
Note that $u|_{z=0} = 0$ since $u_0|_{z=0}=0$.
By integrating by parts and using the boundary condition of $\psi$ and the property of
$\phi$, we have
\begin{eqnarray}
-\int_\Omega \psi_{zt} \Delta \phi \mathd {\Bx}\mathd z-\int_{\Ox} \psi_{zt}\phi_z|_{z=0}\mathd {\Bx}\mathd z-\int_{\Ox} \psi_t \Dx \phi|_{z=0}\mathd {\Bx}\mathd z =\int_\Omega u^2 \phi_{zz}\mathd {\Bx}\mathd z .
\end{eqnarray}
Substituting $\phi$ defined in \myref{phi infty conserved} into the above equation, we have
\begin{eqnarray}
\left(2\left (\frac{\pi}{a}\right)^2-\al^2\right)\frac{\mathd}{\mathd t}\int_\Omega \psi_z \phi \mathd {\Bx}\mathd z = \al^2\int_\Omega u^2 \phi \mathd {\Bx}\mathd z-
2 \left (\frac{\pi}{a}\right)^2 \int_{\Ox}\psi_{t}|_{z=0}\phi_1(\Bx) \mathd {\Bx},
\end{eqnarray}
where $\phi_1(\Bx) =\sin \frac{\pi x_1}{a}\sin \frac{\pi x_2}{a}$.
Finally we have
\begin{eqnarray}
\label{eq psi conserved infty} \frac{\mathd}{\mathd t}\int_\Omega \psi_z \phi
\mathd {\Bx}\mathd z = \frac{\al^2}{2\left (\frac{\pi}{a}\right)^2-\al^2}\int_\Omega u^2 \phi
\mathd {\Bx}\mathd z-\frac{2 \left (\frac{\pi}{a}\right)^2}{2\left (\frac{\pi}{a}\right)^2-\al^2}
\int_{\Ox}\psi_{t}|_{z=0}\phi_1(\Bx)\mathd {\Bx} .
\end{eqnarray}
Next, we multiply $\phi$ to \myref{model-logu} and integrate
over $\Omega$. We get
\begin{eqnarray}
\label{eq u conserved infty} \frac{\mathd}{\mathd t}\int_\Omega (\log u) \phi
\mathd {\Bx}\mathd z=2\int_\Omega \psi_z\phi \mathd {\Bx}\mathd z .
\end{eqnarray}
Combining \myref{eq psi conserved infty} with \myref{eq u conserved
infty}, we obtain
\begin{eqnarray}
\frac{\mathd^2}{\mathd t^2}\int_\Omega (\log u) \phi \mathd {\Bx}\mathd z=
\frac{2\al^2}{2\left (\frac{\pi}{a}\right)^2-\al^2}\int_\Omega u^2
\phi \mathd {\Bx}\mathd z-\frac{4\left (\frac{\pi}{a}\right)^2}{2\left (\frac{\pi}{a}\right)^2-\al^2}
\int_{\Ox}\psi_{t}|_{z=0}\phi_1(\Bx) \mathd {\Bx}
\end{eqnarray}
Integrating the above equation in time and using the assumption that $r(t)\le
\frac{B}{2}$, we get
\begin{eqnarray}
\label{final ineq conserved}
\int_\Omega (\log u) \phi \mathd {\Bx}\mathd z &=& \frac{2\al^2}{2\left (\frac{\pi}{a}\right)^2-\al^2}
\int_0^t\int_0^s\left(\int_\Omega u^2 \phi \mathd {\Bx}\mathd z\right)d\tau ds+A+Bt\nonumber\\
&&-\frac{4\left (\frac{\pi}{a}\right)^2}{2\left (\frac{\pi}{a}\right)^2-\al^2}
\int_0^t\left(\int_{\Ox}(\psi-\psi_{0})|_{z=0}\phi_1(\Bx) \mathd {\Bx}\right) ds\nonumber\\
&\ge&\frac{2\al^2}{2\left (\frac{\pi}{a}\right)^2-\al^2}
\int_0^t\int_0^s\left(\int_\Omega u^2 \phi \mathd {\Bx}\mathd z\right)d\tau ds+A+\frac{1}{2}Bt .
\end{eqnarray}
Using \myref{final ineq conserved}, following the same argument as in
the proof of Theorem \ref{theorem infty mix}, we can prove that the solution of
the initial boundary value problem of the 3D model blows up in a finite time.
\end{proof}

\subsection{Blow-up in a bounded domain}
\label{Neumann-bounded}

In this subsection, we will prove the finite time blow-up of the 3D model with
a conservative boundary condition in a bounded domain.
Specifically, we consider the following initial boundary value problem:
\begin{eqnarray}
\label{model psi bounded conserved neumann}
&&\left\{\begin{array}{rcl}
u_t &=&2u\psi_z\\
-\Delta \psi_t&=&\left(u^2\right)_z
\end{array}\right., \quad (\Bx,z)\in \Omega=\Ox\times(0,b),\\
\label{model psi bounded conserved neumann bc}
&&\psi|_{\p\Omega \backslash\Gamma}=0,\quad \psi_z|_{\Gamma}=0,
\quad
\\
&& \psi|_{t=0}=\psi_0(\Bx,z),\quad u|_{t=0}=u_0(\Bx,z)\ge0,\nonumber
\end{eqnarray}
where $\Bx=(x_1,x_2),\;\Ox=(0,a)\times(0,a),\;
\Gamma=\left\{(\Bx,z)\in \Omega \; |\; \Bx \in \Ox, \;z=0 \;\mbox{or} \;z=b \right\}$.

The main result is stated in the following theorem.
\begin{theorem}
\label{theorem bounded conserved neumann}
Assume that $u_0\in H^2(\Omega)$,
$u_0|_{\p \Omega}=0$ , $u_0|_{\Omega}>0$,
$\psi_0\in H^3(\Omega)$ and satisfies
\myref{model psi bounded conserved neumann bc}.
Let
\begin{eqnarray}
\label{phi conserved bounded neumann}
\phi(\Bx,z)=\frac{e^{-\al(z-b)}-e^{\al(z-b)}}{2}\sin \frac{\pi x_1}{a}
\sin \frac{ \pi x_2}{a},\quad (\Bx,z)\in \Omega,
\end{eqnarray}
with $\al = \frac{\pi}{a}$, and
\begin{eqnarray*}
&&A=\int_\Omega (\log u_0)\phi \mathd {\Bx}\mathd z,\quad B=2\int_\Omega \psi_{0z}\phi \mathd {\Bx}\mathd z,\\
&&r(t)=\frac{2\left(\frac{\pi}{a}\right)^2(e^{\al b}-e^{-\al b})}
{2\left (\frac{\pi}{a}\right)^2-\al^2}
\int_{\Ox} (\psi-\psi_{0})|_{z=0} \sin \frac{\pi x_1}{a}
\sin \frac{ \pi x_2}{a} \mathd {\Bx}\le \frac{B}{2}.
\end{eqnarray*}
If $A>0,\;B>0$ and $r(t)\le \frac{B}{2}$ as long as $u,\psi$ remain regular,
then the solution of \myref{model psi bounded conserved neumann}-\myref{model psi bounded conserved neumann bc}
will develop a finite time singularity in the $H^2$ norm.
\end{theorem}

\begin{proof}
Again, the local well-posedness of \myref{model psi bounded conserved neumann}-\myref{model psi bounded conserved neumann bc} can be established by using
an argument similar to the proof of Theorem \ref{local wellposedness}.
We prove the theorem by contradiction. Assume that the initial boundary
value problem has a globally smooth solution with
$u\in C^1([0,\infty);H^2(\Omega))$
and $\psi \in C^1([0,\infty);H^3(\Omega))$.
Multiplying $\phi_z$ to the both sides of the $\psi$-equation and
integrating over $\Omega$, we have
\begin{eqnarray}
-\int_\Omega \Delta \psi_t \phi_z \mathd {\Bx}\mathd z=\int_\Omega \left(u^2\right)_{z}\phi_z \mathd {\Bx}\mathd z .
\end{eqnarray}
Note that $u|_\Gamma = 0$ as long as the solution remains regular.
By integrating by parts and using the boundary condition of $\psi$ and the property
of $\phi$, we get
\begin{eqnarray}
-\int_\Omega \psi_{zt} \Delta \phi \mathd {\Bx}\mathd z+
\int_{\Ox} \psi_t \Dx \phi|_{z=0}^{z=b}\mathd {\Bx}\mathd z
=\int_\Omega u^2 \phi_{zz}\mathd {\Bx}\mathd z .
\end{eqnarray}
Substituting  $\phi$ to the above equation, we obtain
\begin{eqnarray}
\left(2\left (\frac{\pi}{a}\right)^2-\al^2\right)
\frac{\mathd}{\mathd t}\int_\Omega \psi_z \phi \mathd {\Bx}\mathd z& = &
\al^2 \int_\Omega u^2 \phi \mathd {\Bx}\mathd z \nonumber \\
&&-
\left (\frac{\pi}{a}\right )^2(e^{\al b}-e^{-\al b})
\int_{\Ox}\psi_{t}|_{z=0}\phi_1(\Bx) \mathd {\Bx} ,
\end{eqnarray}
where $\phi_1(\Bx) =\sin \frac{\pi x_1}{a}\sin \frac{\pi x_2}{a}$.
Thus we have
\begin{eqnarray}
\label{eq psi conserved bounded neumann}
\frac{\mathd}{\mathd t}\int_\Omega \psi_z \phi
\mathd {\Bx}\mathd z = \frac{\al^2}{2\left (\frac{\pi}{a}\right)^2-\al^2}
\int_\Omega u^2 \phi \mathd {\Bx}\mathd z-
\frac{\left (\frac{\pi}{a}\right )^2(e^{\al b}-e^{-\al b})}
{2\left (\frac{\pi}{a}\right)^2-\al^2}
\int_{\Ox}\psi_{t}|_{z=0}\phi_1(\Bx) \mathd {\Bx} .
\end{eqnarray}
Next, we multiply $\phi$ to \myref{model-logu} and integrate
over $\Omega$. We have
\begin{eqnarray}
\label{eq u conserved bounded neumann} \frac{\mathd}{\mathd t}\int_\Omega
(\log u) \phi
\mathd {\Bx}\mathd z=2\int_\Omega \psi_z\phi \mathd {\Bx}\mathd z .
\end{eqnarray}
Combining \myref{eq psi conserved bounded neumann} with \myref{eq u conserved
bounded neumann}, we obtain
\begin{eqnarray}
\frac{\mathd^2}{\mathd t^2}\int_\Omega (\log u) \phi \mathd {\Bx}\mathd z=
\frac{2\al^2}{2\left (\frac{\pi}{a}\right)^2-\al^2}
\int_\Omega u^2 \phi \mathd {\Bx}\mathd z-
\frac{2 \left (\frac{\pi}{a}\right )^2(e^{\al b}-e^{-\al b})}
{2\left (\frac{\pi}{a}\right)^2-\al^2}\int_{\Ox}\psi_{t}|_{z=0}\phi_1(\Bx) \mathd {\Bx}.
\quad
\end{eqnarray}
Integrating the above equation in time and using the assumption that $r(t)\le
\frac{B}{2}$, we get
\begin{eqnarray}
\label{final ineq conserved neumann bounded}
\int_\Omega (\log u) \phi \mathd {\Bx}\mathd z &=&
\frac{2\al^2}{2\left (\frac{\pi}{a}\right)^2-\al^2}
\int_0^t\int_0^s\left(\int_\Omega u^2 \phi \mathd {\Bx}\mathd z\right)d\tau ds+A+Bt\nonumber\\
&&-\frac{2 \left (\frac{\pi}{a}\right )^2(e^{\al b}-e^{-\al b})}
{2\left (\frac{\pi}{a}\right)^2-\al^2}
\int_0^t\left(\int_{\Ox}\left(\psi-\psi_0\right)|_{z=0}\phi_1(\Bx)
\mathd {\Bx}\right) ds\nonumber\\
&\ge&\frac{2\al^2}{2\left (\frac{\pi}{a}\right)^2-\al^2}
\int_0^t\int_0^s\left(\int_\Omega u^2 \phi \mathd {\Bx}\mathd z\right)d\tau ds+A+\frac{1}{2}Bt .
\end{eqnarray}
Using \myref{final ineq conserved neumann bounded} and following the same argument as in
the proof of Theorem \ref{theorem infty mix}, we can prove that the solution  of
the 3D model will develop a finite time singularity in the $H^2$ norm.
\end{proof}

\subsection{Blow-up of the 3D model with other conservative boundary conditions}

The singularity analysis we present in the previous subsection can
be generalized to study the finite time blow-up of the 3D model with
the same boundary condition along the $x_1$ and $x_2$ directions as
in Section \ref{Neumann-bounded}, but
changing the Neumann boundary condition along the $z$-direction to
a periodic boundary condition. The assumption on $u_0$ and $\psi_0$
remains the same as in Section \ref{Neumann-bounded}.
In this case, we can prove the finite time
blow-up of the corresponding initial boundary value problem with
two minor modifications in the statement of the blow-up theorem.
The first change is to replace $\phi$ by the following definition:
\begin{eqnarray}
\label{phi periodic bounded }
\phi(\Bx,z)=\frac{e^{-\al z}+e^{-\al(z-b)}}{2}\sin \frac{\pi x_1}{a}
\sin \frac{ \pi x_2}{a},\quad (\Bx,z)\in \Omega =(0,a)\times(0,a)\times (0,b),
\end{eqnarray}
with $\al = \frac{\pi}{a}$. The second change is to modify the definition of
$r(t)$ as follows:
\begin{eqnarray*}
r(t)=\frac{2\al (1-e^{-\al b})}
{2\left (\frac{\pi}{a}\right)^2-\al^2}
\int_{\Ox} (\psi_z-\psi_{0z})|_{z=0} \sin \frac{\pi x_1}{a}
\sin \frac{ \pi x_2}{a} \mathd {\Bx}\le \frac{B}{2},
\end{eqnarray*}
where $A$ and $B$ are the same as in Theorem \ref{theorem bounded conserved neumann}.
If $A>0,\;B>0$ and $r(t)\le \frac{B}{2}$ as long as $u,\psi$ remain regular,
then we can prove that the solution of the corresponding initial boundary value
problem will develop a finite time singularity in the $H^2$ norm.

The same singularity analysis can be applied to
study the finite time blow-up of the 3D model with
the same boundary condition along the $x_1$ and $x_2$ directions as
in Section \ref{Neumann-bounded}, but
changing the Neumann boundary condition along the $z$-direction to
the Dirichlet boundary condition. The assumption on $u_0$ and $\psi_0$
remains the same as in Section \ref{Neumann-bounded}.
In this case, we can prove the finite time
blow-up of the corresponding initial boundary value problem with
two minor modifications in the statement of the blow-up theorem.
The first change is to replace $\phi$ by the following definition:
\begin{eqnarray}
\label{phi periodic bounded-2 }
\phi(\Bx,z)=\frac{e^{\al (z-b)}+e^{-\al(z-b)}}{2}\sin \frac{\pi x_1}{a}
\sin \frac{ \pi x_2}{a},\quad (\Bx,z)\in \Omega =(0,a)\times(0,a)\times (0,b),
\end{eqnarray}
with $\al = \frac{\pi}{a}$. The second change is to modify the definition of
$r(t)$ as follows:
\begin{eqnarray*}
r(t)=\frac{\al (e^{\al b}-e^{-\al b})}
{2\left (\frac{\pi}{a}\right)^2-\al^2}
\int_{\Ox} (\psi_z-\psi_{0z})|_{z=0} \sin \frac{\pi x_1}{a}
\sin \frac{ \pi x_2}{a} \mathd {\Bx}\le \frac{B}{2},
\end{eqnarray*}
where $A$ and $B$ are the same as in Theorem \ref{theorem bounded conserved neumann}.
If $A>0,\;B>0$ and $r(t)\le \frac{B}{2}$ as long as $u,\psi$ remain regular,
then we can prove that the solution of the corresponding initial boundary value
problem will develop a finite time singularity in the $H^2$ norm.

\begin{remark}
All the results in this section can be generalized to a cylindrical domain $\Omega$ in
high dimension space $\mathbb{R}^N$, with
$\Omega=\left\{(\mathbf{x},z)|\; \mathbf{x}\in \Omega_\mathbf{x}\subset \mathbb{R}^{N-1}, z\in [a,b]\subset \mathbb{R} \right\}$. In this case, the weight
function $\phi(\mathbf{x},z)$ is chosen to be the product of two functions:
\begin{eqnarray}
\phi(\mathbf{x},z)=\phi_1(\mathbf{x})\eta(z).
\end{eqnarray}
Here the eigen-function, $\eta(z)$, is the same in the Eulerian coordinate in
the previous sections. The eigen-function, $\phi_1(\Bx)$, defined in the
$\Bx$ space, is chosen to be the first eigen-function of the
following eigenvalue problem:
\begin{eqnarray}
-\Delta_\mathbf{x}  \phi_1&=&\lambda \phi_1,\\
\phi_1|_{\p \Omega_{\mathbf{x}}}&=&0,
\end{eqnarray}
with $\lambda >0$, where $\Delta_\mathbf{x}$ is the $N-1$ dimensional Laplace
operator,
$\Delta_\mathbf{x}=\frac{\p^2}{\p x_1^2}+\cdots+\frac{\p^2}{\p x_{N-1}^2}$.
\end{remark}

\section{Global regularity of the 3D inviscid model with small data }

In this section, we will prove the global regularity of the 3D
inviscid  model for a class of small initial data with some
appropriate boundary condition. We remark that since we consider
the inviscid version of the 3D model, there is no viscosity in
the model equation. Although we impose some smallness condition
on the initial data, such result is still very interesting
since there is currently no global regularity result for
the 3D incompressible Euler equations even for small initial data.

To simplify
the presentation of our analysis, we use $u^2$ and $\psi_z$ as our
new variables. We will define $v=\psi_z$ and still use $u$ to stand
for $u^2$.  Then the 3D model now has the form:
\begin{eqnarray}
\label{model regular}
&&\left\{\begin{array}{rcl}
u_t &=&4uv\\
-\Delta v_t&=&u_{zz}
\end{array}\right., \quad (\Bx,z)\in
\Omega=(0,\delta)\times(0,\delta)\times(0,\delta).
\end{eqnarray}
We choose the following boundary condition for $v$:
\begin{eqnarray}
\label{Global-BC}
v|_{\p \Omega}=-4,
\end{eqnarray}
and denote $ v|_{t=0}=v_0(\Bx,z)$ and $u|_{t=0}=u_0(\Bx,z)\ge0$.

In our regularity analysis, we need to use the following Sobolev
inequality \cite{Foland95}:
%\begin{lemma}
%\begin{eqnarray}
%\label{reisz regular} \|v_t\|_{H^s}\le C_s \|u\|_{H^s}
%\end{eqnarray}
%\end{lemma}

\begin{lemma}
\label{lemma majda} For all $s\in \mathbb{Z}^+$, there exists
$C_s>0$, such that, for all $u,v \in L^{\infty}\cap
H^s\left(\mathbb{R}^N\right)$,
\begin{eqnarray}
\left(\sum_{0\le |\al| \le s} \|\p^\al (uv)-\p^\al u\cdot
v\|_{L^2}^2\right)^{1/2}\le
C_s\left(\|u\|_{L^{\infty}}\|v\|_{H^s}+\|\nabla
v\|_{L^\infty}\|u\|_{H^{s-1}}\right) .
\end{eqnarray}
\end{lemma}

Now we state the main result of this section.
\begin{theorem}\label{lemma-small}
Assume that $u_0, v_0\in H^s(\Omega)$ with $s \ge 4$, $u_0|_{\p
\Omega} = 0$, $v_0|_{\p \Omega} = -4$ and
$v_0\le -4$ over $\Omega$, then the solution
of \myref{model regular}-\myref{Global-BC} remains regular in
$H^s(\Omega)$ for all time as long as the following holds
\begin{eqnarray}
\label{IC-est-global} \delta (4C_s+1)
\left(\|v_0\|_{H^s}+C_s\|u_0\|_{H^s}\right)<1,
\end{eqnarray}
where $C_s$ is an interpolation constant. Moreover, we have
$\|u\|_{L^\infty}\le
\|u_0\|_{L^\infty}e^{-7t},\;\|u\|_{H^s(\Omega)}\le
\|u_0\|_{H^s(\Omega)}e^{-7t}$ and $\|v\|_{H^s(\Omega)}\le C$ for
some constant $C$ which depends on $u_0, v_0$ and $s$ only.
\end{theorem}
\begin{proof}
First of all, we note that $v_t$ satisfies the homogeneous
boundary condition on $\partial \Omega$ since $v = -4$
on $\partial \Omega$.
Let $K = (-\Delta)^{-1} $ be the inverse Laplacian
operator with homogeneous Dirichlet boundary condition. Then, we
can rewrite (\ref{model regular}) as follows:
\begin{eqnarray}
\label{model-regular-2} &&\left\{\begin{array}{rcl}
u_t &=&4uv\\
v_t&=&K(u_{zz})
\end{array}\right., \quad (\Bx,z)\in
\Omega=(0,\delta)\times(0,\delta)\times(0,\delta).
\end{eqnarray}
Standard elliptic theory implies that $K$ is a
linear bounded operator from $H^{s-2}(\Omega)$ to
$H^s(\Omega)$, that is for any $f\in H^{s-2}(\Omega)$, we have
\begin{eqnarray}
\label{bound k-2}
  \|K(f)\|_{H^s(\Omega)}\le  C_s\|f\|_{H^{s-2}(\Omega)},
\end{eqnarray}
for $s \geq 2$. Such estimate can be also obtained directly
by using an argument similar to the proof of Lemma \ref{reisz}.

Next, we define
 $V^s=\{v\in H^s(\Omega): v|_{\partial\Omega}=-4\}$.
Since $s \geq 4$, the trace of $v$ on $\partial \Omega$ is well-defined.
Let $X:=H^s(\Omega) \times V^s(\Omega)$ be a Banach space
with the norm
$\|\cdot\|_X$ of the space $X$ defined as follows:
\begin{eqnarray*}
  \|U\|_X=\|U_1\|_{H^s(\Omega)}+\|U_2\|_{H^s(\Omega)}.
\end{eqnarray*}
Further we express the system (\ref{model-regular-2})
as an ODE in the Banach space $X$:
\begin{equation}
\label{new-system-2}
U_t = F(U),
\end{equation}
where $U = (U_1, U_2)= (u,v)$ and $F(U) = (F_1(U),
F_2(U))=\left(4uv, K(u_{zz})\right)$.

We note that $K \partial_{zz}$ is a bounded linear
operator from $H^s(\Omega)$ to $H^s(\Omega)$.
By using an argument similar to the
local well-posedness analysis presented in Section 2.1, we can
show that the system (\ref{model-regular-2}) is locally well-posed
and there exists $T_0 > 0$ such that $\|u\|_{H^s}$ and
$\|v\|_{H^s}$ are bounded for $0 \leq t \leq T_0$.
Furthermore, by using Lemma \ref{lemma hs} and the fact that
$K \partial_{zz}$ is a bounded operator from $H^s$
to $H^s$, we can easily obtain the following {\it a priori}
estimate
\[
\frac{d}{dt} \|U \|_X \leq C_s \|U\|_X^2 ,
\]
for $ 0 \leq t \leq T_0$, which implies that $\|U\|_X$ is bounded
by a constant $M$ that depends on $\|U_0\|_X$ only for
$0 \leq t \leq \overline{T}_0 < \min(T_0, 1/(C_s \|U_0\|_X))$.

On the other hand, since $K \partial_{zz}$ is a bounded operator from $H^s$
to $H^s$, we obtain by standard energy estimates that
\[
\frac{d}{dt} \|v\|_{H^s(\Omega)} \leq C_s \|u\|_{H^s(\Omega)}
\leq C_s M(\|u_0\|_{H^s(\Omega)},\|v_0\|_{H^s(\Omega)}),
\]
from which we conclude that $\|v(t)\|_{H^s(\Omega)}$ can be
made as close to $\|v_0\|_{H^s(\Omega)}$ as we wish for
$0 \leq t \leq \overline{T}_0$ by making
$\overline{T}_0$ small enough.
Similarly, since $s \ge 4$, we have by using
the Sobolev embedding theorem and the {\it a priori} estimate that
\[
\|v_t\|_{L^\infty(\Omega)} \leq C_0
\|v_t\|_{H^s(\Omega)} \leq C_0 \|K(u_{zz}) \|_{H^s(\Omega)}
\leq C_s M(\|u_0\|_{H^s(\Omega)},\|v_0\|_{H^s(\Omega)}).
\]
Thus we can also make $\|v(t)-v_0\|_{L^\infty(\Omega)}$ as small
as we wish for $0 \leq t \leq \overline{T}_0$ by making
$\overline{T}_0$ small enough.

Note that \myref{IC-est-global} implies that $2C_s \delta
\|v_{0} \|_{H^s} < \frac{1}{2}$. By our assumption, we also have
$v_0 \leq -4$ in $\Omega$. Based on the above argument, we can choose
$\overline{T}_0$
small enough so that we have $v(t) < - 2$
on $\Omega$, and $ 2C_s \delta \|v(t) \|_{H^s} < 1 $ for $ 0 \leq t <
\overline{T}_0$.

Let $\left[ 0, T \right)$ be the largest time interval on which
$\|u\|_{H^s}$ and $\|v\|_{H^s}$ are bounded, and both of the
following inequalities hold:
%\begin{itemize}
%\item[1.] $ v\le -2$ over $\Omega$,
%\item[2.] $2C_s\delta \|v\|_{H^s}\le 1.$
%\end{itemize}
\[
v\le -2\;\; \mbox{over} \;\; \Omega, \quad 2C_s\delta \|v\|_{H^s}\le
1.
\]
We will show that $T=\infty$.

For $\al =(\al_1,\al_2,\al_3)$ with $\al_j \ge 0$ ($j=1,2,3$) and
$|\al| \le s$, we have for $0 \le t < T$ that
\begin{eqnarray}
\frac{\mathd}{\mathd t}\left<\p^\al u, \p^\al u\right>&=&8\left<\p^\al (uv), \p^\al u\right>\nonumber\\
&=&8\left<\p^\al u\cdot v, \p^\al u\right>+8\left<\p^\al (uv)-\p^\al u\cdot v, \p^\al u\right>\nonumber\\
&=&8\int_\Omega |\p^\al u|^2v\mathd {\Bx}\mathd z+8\left<\p^\al (uv)-\p^\al u\cdot v, \p^\al u\right>\nonumber\\
&\le&-16\int_\Omega |\p^\al u|^2\mathd {\Bx}\mathd z+8\|\p^\al
(uv)-\p^\al u\cdot v\|_{L^2}\|\p^\al u\|_{L^2} .
\end{eqnarray}
Using Lemma \ref{lemma majda}, we get
%\begin{eqnarray}
%\frac{\mathd}{\mathd t}\|u\|^2_{H^s}\le -4\|u\|_{H^s}^2+2C_s\|u\|_{H^s}\left(\|u\|_{L^{\infty}}\|v\|_{H^s}+\|\nabla v\|_{L^\infty}\|u\|_{H^{s-1}}\right)
%\end{eqnarray}
\begin{eqnarray}
\label{eq u hs}
 \frac{\mathd}{\mathd t}\|u\|_{H^s}\le
-8\|u\|_{H^s}+C_s\left(\|u\|_{L^{\infty}}\|v\|_{H^s}+\|\nabla
v\|_{L^\infty}\|u\|_{H^{s-1}}\right) .
\end{eqnarray}
Since $u|_{\p \Omega} = u_0|_{\p \Omega} = 0$, we obtain
\begin{eqnarray}
\label{u l infty}
u(\Bx,z,t)&=&\int_0^z\p_{z'}u(\Bx,z',t)\mathd z'\nonumber\\
&=&\int_0^z\int_0^{x_1}\int_0^{x_2}\p_{x_1'}\p_{x_2'}\p_{z'}u(x_1',x_2',z',t)
\mathd {x_1}' \mathd {x_2}'
\mathd z'\nonumber\\
&\le& \delta^{3/2} \| \p_{x_1}\p_{x_2}\p_z u\|_{L^2} \le \delta
\|u\|_{H^s},\quad
\end{eqnarray}
since $s \ge 4$. Notice that $v_{x_i}|_{z=0}=0$, so we have
\begin{eqnarray}
\label{vx-inf} v_{x_i}=\int_0^z v_{x_iz'}\mathd z'\le \int_0^\delta
|v_{x_iz'}|\mathd z'\le \delta \|v_{x_iz}\|_{L^{\infty}} .
\end{eqnarray}
Similarly, since $v_z|_{x_1=0} = 0$, we have
\begin{eqnarray}
\label{vz-inf} v_z=\int_0^{x_1} v_{x_1'z}\mathd {x_1}'\le
\int_0^\delta |v_{x_1'z}|\mathd {x_1}'\le \delta
\|v_{x_1z}\|_{L^{\infty}} .
\end{eqnarray}
Combining \myref{vx-inf} with \myref{vz-inf}, we get
\begin{eqnarray}
\label{dv-inf} \|\nabla v\|_{L^\infty}\le \delta \max_{i=1,2}
(\|v_{x_iz}\|_{L^{\infty}}) .
\end{eqnarray}
Since $s \ge 4 > 2+ 3/2$ by our assumption, we obtain by using the
Sobolev embedding theorem \cite{Foland95} that
\begin{eqnarray}
\label{vxiz-inf} \|v_{x_iz}\|_{L^\infty}\le C_s
\|v_{x_iz}\|_{H^{s-2}} \le C_s \|v\|_{H^s}.
\end{eqnarray}
It follows from \myref{dv-inf} and \myref{vxiz-inf} that
\begin{eqnarray}
\label{v l infty} \|\nabla v\|_{L^\infty}\le C_s\delta \|v\|_{H^{s}}
.
\end{eqnarray}
Combine \myref{eq u hs}-\myref{u l infty} with \myref{v l infty}, we
obtain
\begin{eqnarray}
\frac{\mathd}{\mathd t}\|u\|_{H^s}\le
\left(-8+2C_s\delta\|v\|_{H^{s}}\right)\|u\|_{H^{s}}.
\end{eqnarray}
Since $2C_s\delta\|v\|_{H^{s}} \le 1$ for $ t < T$ by the assumption
of $T$, we have for $ t < T$ that
\begin{eqnarray}
\label{u hs} \|u\|_{H^s}\le \|u_0\|_{H^s}e^{-7t}.
\end{eqnarray}
Note that
\begin{eqnarray}
\frac{\mathd}{\mathd t}\left<\p^\al v, \p^\al v\right>=2\left<\p^\al
v_t, \p^\al v\right>\le 2\|\p^\al v_t\|_{L^2}\|\p^\al v\|_{L^2} .
\end{eqnarray}
Recall that $\Delta v_t = u_{zz}$. We can easily generalize the
proof of Lemma \ref{reisz} to show that
\begin{equation}
\label{vt-reisz} \|v_t \|_{H^s(\Omega} \le C_s \|u_{zz}\|_{H^{s-2}
(\Omega)} \le C_s \|u\|_{H^s(\Omega)}.
\end{equation}
Using \myref{vt-reisz}, we get
\begin{eqnarray}
\frac{\mathd}{\mathd t}\|v\|_{H^s}^2\le 2\|v_t\|_{H^s}\|v\|_{H^s}\le
2C_s \|u\|_{H^s}\|v\|_{H^s}.
\end{eqnarray}
Substituting \myref{u hs} to the above equations, we get for $t < T$
that
\begin{eqnarray}
\frac{\mathd}{\mathd t}\|v\|_{H^s}\le C_s \|u\|_{H^s}\le
C_s\|u_0\|_{H^s}e^{-7t} .
\end{eqnarray}
Integrating the above inequality in time, we obtain the estimate of
$\|v\|_{H^s}$ over $[0,T)$:
\begin{eqnarray}
\|v\|_{H^s}&\le & \|v_0\|_{H^s}+C_s\|u_0\|_{H^s}\int_0^t e^{-7s}ds\nonumber \\
&\le& \|v_0\|_{H^s}+\frac{C_s}{7} \|u_0\|_{H^s} \le
\|v_0\|_{H^s}+C_s \|u_0\|_{H^s}  . \label{vHs}
\end{eqnarray}
Since $v|_{\p \Omega} = -4$, we can use the the same argument as in
the proof of \myref{u l infty} to show that
\begin{eqnarray}
|v+4|\le \delta \|v\|_{H^s}\le
\delta\left(\|v_0\|_{H^s}+C_s\|u_0\|_{H^s}\right),
\end{eqnarray}
where we have used \myref{vHs}. Now we have for $t < T$ that
\begin{eqnarray*}
&& v\le -4+\delta\left(\|v_0\|_{H^s}+C_s\|u_0\|_{H^s}\right),\\
\vspace{0.1in} &&2C_s\delta \|v\|_{H^s}\le
2C_s\delta\left(\|v_0\|_{H^s}+C_s\|u_0\|_{H^s}\right).
\end{eqnarray*}
By our assumption on the initial data, we have
\begin{eqnarray}
\delta (4C_s+1) \left(\|v_0\|_{H^s}+C_s\|u_0\|_{H^s}\right)<1.
\end{eqnarray}
Therefore, we have proved that if
\begin{eqnarray}
\label{interval cond} v\le-2\; \mbox{on}\; \Omega \quad
\mbox{and}\quad 2C_s\delta \|v\|_{H^s}\le 1, \quad 0\le t<T,
\end{eqnarray}
then we actually have
\begin{eqnarray}
 v\le-3\; \mbox{on}\; \Omega \quad \mbox{and}\quad 2C_s\delta \|v\|_{H^s}\le \frac{1}{2}, \quad 0\le t<T.
\end{eqnarray}
This implies that we can extend the time interval beyond $T$ so that
\myref{interval cond} is still valid. This contradicts the
assumption that $[0,T)$ is the largest time interval on which
\myref{interval cond} is valid. This contradiction shows that $T$
can not be a finite number, i.e. \myref{interval cond} is true for
all time. This implies that $\|u\|_{H^s(\Omega)}$ and
$\|v\|_{H^s(\Omega)}$ are bounded for all time. Moreover, we have
shown that $\|u\|_{L^\infty}\le
\|u_0\|_{L^\infty}e^{-7t},\;\|u\|_{H^s(\Omega)}\le
\|u_0\|_{H^s(\Omega)}e^{-7t}$ and $\|v\|_{H^s(\Omega)}\le
\|v_0\|_{H^s}+C_s\|u_0\|_{H^s}$.
\end{proof}

  \renewcommand{\theequation}{A-\arabic{equation}}
  % redefine the command that creates the equation no.
  %\setcounter{section}{0}
  \setcounter{equation}{0}  % reset counter
  \section*{Appendix A. }

\begin{proof} \textbf{of Lemma \ref{reisz}}
We present the proof for the case of $a = \pi$. The case of
$a \neq \pi$ can be proved similarly.
First, we perform the sine transform
along $x_1$ and $x_2$ directions to both sides of \myref{eqn-Laplace}.
We have
\begin{eqnarray}
\label{eqn-Laplace-trans}
|k|^2\hat{v}(k,z)-\hat{v}_{zz}(k,z)=\hat{f}(k,z),
\end{eqnarray}
where $k=(k_1,k_2)$, $|k| = \sqrt{k_1^2+k_2^2}$,
and the sine transform of $v$ is defined as follows:
\begin{eqnarray}
\label{sin-tranf}
\hat{v}(k,z)=\left(\frac{2}{\pi}\right)^2
\int_0^{\pi} \int_0^{\pi} v(x_1,x_2,z)\sin
(k_1  x_1) \sin ( k_2 x_2) dx_1 dx_2 .
\end{eqnarray}
Applying the sine transform to the boundary condition gives
\begin{eqnarray}
\label{ODE-BC}
(\hat{v}_z(k,z)+\beta\hat{v}(k,z))|_{z=0}=0.
\end{eqnarray}
The second order ODE \myref{eqn-Laplace-trans}
can be solved analytically.
The general solution is given by
\begin{eqnarray}
\label{vhat}
\hat{v}(k,z)=\frac{e^{|k|z}}{|k|}\left(-\frac{1}{2}\int_0^z\hat{f}
e^{-|k|z'}dz'+C_1(k)\right)+\frac{e^{-|k|z}}{|k|}
\left(\frac{1}{2}\int_0^z\hat{f} e^{|k|z'}dz'+C_2(k)\right) .
\end{eqnarray}
The boundary condition \myref{ODE-BC} and the constraint that
$v \in L^2(\Omega)$ determine the constants
$C_1$ and $C_2$ uniquely as follows:
\begin{eqnarray}
\label{C1-C2}
C_1(k)=\frac{1}{2}\int_0^\infty\widehat{f}(k,z) e^{-|k|z'}dz',
\quad C_2(k)=\frac{|k|+\beta}{|k|-\beta}C_1(k) .
\end{eqnarray}
%\begin{eqnarray}
%\hat{v}(k,z)=-\frac{1}{2k}\int_z^\infty\hat{f} e^{-k(z'-z)}dz'+\frac{1}{2k}\int_0^z\hat{f} e^{-k(z-z')}dz'+\frac{1}{k}%\int_0^\infty\hat{f} e^{-k(z+z')}dz'.
%\end{eqnarray}
Let $\chi(x)$ be the characteristic function
\begin{eqnarray}
\chi(x)=\left\{\begin{array}{ll}
0,& x\le0,\\
1,&x> 0.
\end{array}\right.
\end{eqnarray}
Then $\widehat{v}(k,z)$ has the following integral representation
(note that $\beta \neq |k|$ by our assumption):
\begin{eqnarray}
\widehat{v}(k,z)&=&-\frac{1}{2|k|}\int_0^\infty\widehat{f} (k,z)
e^{-|k|(z'-z)}\chi(z'-z)dz'+\frac{1}{2|k|}\int_0^\infty
\widehat{f} (k,z) e^{-|k|(z-z')}\chi(z-z')dz'\nonumber\\
&&+\frac{|k|+\beta}{|k|(|k|-\beta)}\int_0^\infty\widehat{f}(k,z)
e^{-|k|(z+z')}dz'\nonumber\\
&=&-\frac{1}{2|k|}\int_0^\infty\widehat{f}(k,z) K_1(z'-z)dz'+
\frac{1}{2|k|}\int_0^\infty\widehat{f}(k,z) K_1(z-z')dz'\nonumber \\
&&
\label{vhat-integral}
+ \frac{|k|+\beta}{|k|(|k|-\beta)}\int_0^\infty\widehat{f} (k,z)
K_2(z+z')dz',\quad \quad \quad
\end{eqnarray}
where $K_1(z)=e^{-|k|z}\chi(z),\; K_2(z)=e^{-|k|z}$.
Using Young's inequality (see e.g. page 232 of \cite{Foland84}), we obtain:
\begin{eqnarray}
\label{vhat l2}
\|\widehat{v}(k,\cdot)\|_{L^2[0,\infty)}&\le& \frac{1}{2|k|}
\left(2\|K_1\|_{L^1[0,\infty)}+
2\left|\frac{|k|+\beta}{|k|-\beta}\right|\|K_2\|_{L^1[0,\infty)}
\right)\|\widehat{f}(k,\cdot)\|_{L^2[0,\infty)}\nonumber\\
&\le&
\frac{1}{|k|^2}\left(1+\left|\frac{|k|+\beta}{|k|-\beta}
\right|\right)\|\widehat{f}(k,\cdot)\|_{L^2[0,\infty)}
\le \frac{M}{|k|^2}\|\widehat{f}(k,\cdot)\|_{L^2[0,\infty)} ,
\end{eqnarray}
where $\D M=\max_{k_1, k_2>0}\left(1+\left|\frac{|k|+\beta}{|k|-\beta}\right|\right)<\infty$
since $\beta \neq |k| $ for any $k\in \mathbb{Z}^2$ by our assumption.

Next, we estimate $\widehat{v}_z(k,z)$. Differentiating \myref{vhat} with
respect to $z$, we get
\begin{eqnarray}
\widehat{v}_z(k,z)&=&-\frac{1}{2}\int_z^\infty\widehat{f} (k,z)
e^{-|k|(z'-z)}dz'-\frac{1}{2}\int_0^z\widehat{f}(k,z) e^{-|k|(z-z')}dz'
\nonumber \\
&&-
\frac{|k|+\beta}{|k|-\beta}\int_0^\infty\widehat{f}(k,z) e^{-|k|(z+z')}dz'.
\end{eqnarray}
Following the same procedure as in our estimate for $\widehat{v}(k,z)$,
we obtain a similar estimate for $\widehat{v}_z(k,z)$:
\begin{eqnarray}
\label{vzhat l2}
\|\widehat{v}_z(k,\cdot)\|_{L^2[0,\infty)}\le
\frac{1}{|k|}\left(1+\left|\frac{|k|+\beta}{|k|-\beta}
\right|\right)\|\widehat{f}(k,\cdot)\|_{L^2[0,\infty)}\le
\frac{M}{|k|}
\|\widehat{f}(k,\cdot)\|_{L^2[0,\infty)}.
\end{eqnarray}
Let $\al=(\al_1,\al_2,\al_3)$ with $\al_j \ge 0$ ($j=1,2,3$).
We will prove
$\|\p^{\al} v\|^2_{L^2}\leq M^2 \|f\|^2_{H^{|\al|-2}}$ for
all $|\al| \ge 2$. We will prove this using an induction argument
on $\al_3$. First, we establish this estimate for $\al_3=0$
and $\al_1+\al_2 \ge 2$. Below we use the case of
$\al_1\geq 1$ and $\al_2\geq 1$ as an example to illustrate
the main idea. By using the Parseval equality and \myref{vhat l2},
we obtain
\begin{eqnarray}
\|\p^{\al} v\|^2_{L^2(\Omega)}&=&
\sum_{k_1=1}^\infty
\sum_{k_2=1}^\infty k_1^{2\al_1}k_2^{2\al_2}
\int_0^\infty |\widehat{v}(k,z)|^2 dz \nonumber \\
&=& \sum_{k_1=1}^\infty
\sum_{k_2=1}^\infty k_1^{2\al_1}k_2^{2\al_2}
\|\widehat{v}(k,\cdot)\|_{L^2[0,\infty)}^2\nonumber\\
&\le&\sum_{k_1=1}^\infty \sum_{k_2=1}^\infty
 M^2k_1^{2\al_1}k_2^{2\al_2} |k|^{-4}
\|\widehat{f}(k,\cdot)\|^2_{L^2[0,\infty)}\nonumber\\
&\le&M^2\sum_{k_1=1}^\infty \sum_{k_2=1}^\infty
\|k_1^{\al_1-1}k_2^{\al_2-1}\widehat{f}(k,\cdot)\|^2_{L^2[0,\infty)}\nonumber\\
\label{estimate-A11}
&=&M^2\|\p_x^{\al_1-1}\p_y^{\al_2-1} f\|^2_{L^2(\Omega)}\le
M^2 \|f\|^2_{H^{|\al|-2}(\Omega)}.
\end{eqnarray}
Similarly, we can prove \myref{estimate-A11} for $\al_3=0$ and
$\al_1 +\al_2 \ge 2$ by distributing the appropriate order of
derivatives to $x_1$ and/or $x_2$ direction.

Using \myref{vzhat l2} and following the same procedure as
in the proof of \myref{estimate-A11}, we can prove
\myref{estimate-A11} for the case of $\al_3 =1 $ and
$\al_1 + \al_2 \ge 1$. Finally, using
\myref{eqn-Laplace-trans} and differentiating
\myref{eqn-Laplace-trans} with respect to $z$ as many times
as needed, we can prove
\begin{equation}
\label{estimate-A12}
\|\p^{\al} v\|^2_{L^2(\Omega)} \le C_\al \|f\|^2_{H^{|\al|-2}(\Omega)},
\end{equation}
for all $\al_3 \ge 2 $ and $\al_1+\al_2 \ge 0$ by using
an induction argument and \myref{estimate-A11} for
$\al_3 =0 $ and $\al_3=1$.
Using \myref{estimate-A12} and \myref{vhat-integral}, we obtain
\begin{eqnarray}
\|v\|_{H^s(\Omega)}\le C_s\|f\|_{H^{s-2}(\Omega)},
\end{eqnarray}
for all $s \ge 2$, where $C_s$ is a constant depending only on $s$.
The uniqueness of the solution follows from the solution formula
\myref{vhat} and \myref{C1-C2}.
This completes the proof of Lemma \ref{reisz}.
\end{proof}

\vspace{0.1in} \noindent {\bf Acknowledgments} Dr. T. Hou would
like to acknowledge NSF for their generous support through the
Grants DMS-0713670 and DMS-0908546. The work of Drs. Z. Shi and S.
Wang was supported in part by the NSF grant DMS-0713670. The
research of Dr. S. Wang was also supported by China 973
Program(Grant no. 2011CB808002), the Grants NSFC 11071009 and
PHR-IHLB 200906103. This work was done during Dr. Shu Wang's visit
to ACM at Caltech. He would like to thank Prof. T. Hou and Caltech
for their hospitality during his visit. We would like to thank
Profs. Joseph Keller, Congming Li, Xinwei Yu, and the referee
for their comments which help improve the quality of this work.

\end{document}